\documentclass[reqno]{amsart}

\usepackage{amsmath}
\usepackage{amsthm}
\usepackage{amssymb}
\usepackage{graphics}
\usepackage{latexsym}
\usepackage{comment}
\usepackage{extarrows}
\usepackage{wrapfig}
\usepackage{comment}
\usepackage {colortbl,array,xcolor}
\usepackage{hhline}
\usepackage{tikz}
\usepackage{hyperref}
\usepackage{enumitem}
\usetikzlibrary{intersections, calc, math, patterns}

\numberwithin{equation}{section}
\newtheorem{thm}{Theorem}[section]
\newtheorem{prop}[thm]{Proposition}
\newtheorem{lem}[thm]{Lemma}

\newtheorem{defn}[thm]{Definition}
\newtheorem{claim}{Claim}
\theoremstyle{remark}
\newtheorem{rem}[thm]{Remark}
\newtheorem{assumption}{Assumption}

\newcommand{\R}{{\mathbb R}}
\newcommand{\Z}{{\mathbb Z}}

\newcommand{\N}{{\mathbb N}}
\newcommand{\C}{{\mathbb C}}
\newcommand{\Sp}{{\mathbb S}}

\newcommand{\F}{\mathcal{F}}

\newcommand{\supp}{\operatorname{supp}}

\date{\today}

\title[Well-posedness for quadratic $2$D NLS]{Sharp well-posedness for the Cauchy problem of the two dimensional quadratic nonlinear Schr\"{o}dinger equation with angular regularity}

\author[H. Hirayama]{Hiroyuki Hirayama}
\address[H. Hirayama]{Faculty of Education, University of Miyazaki,
1-1, Gakuenkibanadai-nishi, Miyazaki, 889-2192 Japan}
\email[H. Hirayama]{h.hirayama@cc.miyazaki-u.ac.jp}

\author[S. Kinoshita]{Shinya Kinoshita}
\address[S. Kinoshita]{Department of Mathematics,
Graduate School of Science and Engineering,
Saitama University,
255 Shimo-Okubo,
Saitama 338-8570, Japan}
\email[S. Kinoshita]{kinoshita@mail.saitama-u.ac.jp}

\author[M. Okamoto]{Mamoru Okamoto}
\address[M. Okamoto]{Department of Mathematics,
Graduate School of Science, Osaka University,
Toyonaka, Osaka 560-0043, Japan}
\email[M. Okamoto]{okamoto@math.sci.osaka-u.ac.jp}

\numberwithin{equation}{section}

\begin{document}

\subjclass[2020]{35Q55, 35B30}

\keywords{Schr\"odinger equation;
Cauchy problem;
well-posedness;
low regularity}

\begin{abstract}
This paper is concerned with the Cauchy problem of the quadratic nonlinear Schr\"{o}dinger equation in $\R \times \R^2$ with the nonlinearity $\eta |u|^2$ where $\eta \in \C \setminus \{0\}$ and low regularity initial data. 
If $s < -1/4$, the ill-posedness result in the Sobolev space $H^{s}(\R^2)$ is known. We will prove the well-posedness in $H^s(\R^2)$ for $-1/2 < s < -1/4$ by assuming some angular regularity on initial data. 
The key tools are the modified Fourier restriction norm and the convolution estimate on thickened hypersurfaces.
\end{abstract}

\maketitle
\section{Introduction}
We consider the Cauchy problem of the nonlinear Schr\"{o}dinger equation (NLS):
\begin{equation}\label{NLS}
\begin{cases}
i\partial_{t} u (t,x) - \Delta u(t,x) =\eta |u(t,x)|^2, \quad (t,x) \in \R \times \R^d,\\
u(0,x)=u_0(x),
\end{cases}
\end{equation}
where $\eta \in \C \setminus \{0\}$, $\Delta$ is the Laplacian in $\R^d$, and the solution $u$ is complex-valued. The following scaling transformation
\[
u_{\lambda}(t,x) = \lambda^2 u(\lambda^2 t, \lambda x), \quad \lambda >0
\]
implies that \eqref{NLS} has a scaling invariance in the homogeneous Sobolev space $\dot{H}^{s_c}(\R^d)$ with the scaling critical index $s_c =d/2-2$. 
Our aim is to establish the local in time well-posedness of \eqref{NLS} with $d=2$ in a low regularity Sobolev space $H^s(\R^2)$ with an additional angular regularity. 

There are a lot of works on the Cauchy problem of NLS with quadratic nonlinearities. 
It is known that a necessary regularity for the well-posedness depends heavily 
on the structure of the nonlinear term. To see this, let us first consider the Cauchy problem of NLS with the nonlinear term $\eta u^2$ where $\eta \in \C \setminus \{0\}$:
\begin{equation}\label{eq:NLSu2}
i\partial_{t} u - \Delta u =\eta u^2 \quad \textnormal{in} \ \R \times \R^d.
\end{equation}
In the one dimensional case, Kenig, Ponce, and Vega \cite{KPV96} proved that the Cauchy problem of \eqref{eq:NLSu2} is locally well-posed in $H^s(\R)$ if $s>-3/4$ by exploiting the Fourier restriction norm method. 
Bejenaru and Tao \cite{BT06} relaxed the regularity condition by showing the local well-posedness of \eqref{eq:NLSu2} in $H^s(\R)$ for $s \geq -1$ and, in addition, they proved ill-posedness if $s$ is below $-1$. 
In the two dimensional case, the local well-posedness for \eqref{eq:NLSu2} in $H^s(\R^2)$ with $s>-3/4$ was obtained in \cite{CDKS01}. This range was extended to $s > -1$ in \cite{BS08}. On the other hand, Iwabuchi and Ogawa \cite{IO15} found that \eqref{eq:NLSu2} is ill-posed in the scaling critical space $H^{-1}(\R^2)$ by showing so-called ``norm inflation''. 

Note that the same well-posedness result for the Cauchy problems of NLS with the nonlinear term $\eta_1 u^2 + \eta_2 \overline{u}^2$ where $\eta_1$, $\eta_2 \in \C \setminus \{0\}$ was proved in \cite{Kishi09}.

Now we turn to \eqref{NLS}. For $d=1$, Kenig, Ponce, and Vega \cite{KPV96} proved the local well-posedness of \eqref{NLS} in $H^s(\R)$ for $s>-1/4$. This result was improved by Kishimoto and Tsugawa \cite{KT10} who showed the well-posedness in $H^{-\frac14}(\R)$. 
In the two dimensional case, Colliander et al. \cite{CDKS01} proved that \eqref{NLS} is locally well-posed if $s >-1/4$. This range was extended to $s \geq -1/4$ by Kishimoto~\cite{Kishi08}. 
In the three dimensional case, Tao \cite{Tao01} established the local well-posedness for $s>-1/4$. Later, it was observed by Iwabuchi and Uriya \cite{IU15} that the regularity $s=-1/4$ is critical for the well-posedness. They proved that the norm inflation occurs in $H^s(\R^d)$ when $d=1$, $2$, $3$ if $s <-1/4$. Especially, \eqref{NLS} is ill-posed in $H^s(\R^d)$ when $d=1$, $2$, $3$, and $s<-1/4$.

Our goal in this paper is to prove the local well-posedness of \eqref{NLS} with $d=2$ in $H^s(\R^2)$ for $s>-1/2$ under some angular regularity assumption on initial data. As mentioned above, the necessary regularity is completely determined, that is, $s=-1/4$ is the threshold for the well-posedness of \eqref{NLS} in the $H^s(\R^2)$ framework with no extra angular regularity. We will push down the threshold to $s>-1/2$ by exploiting angular regularity of initial data.

In the study of the Cauchy problem of NLS with general $p$-power type nonlinearities, e.g. 
$\eta |u|^{p-1}u$, $\eta |u|^p$, $\eta u^p$, radially symmetry and angular regularity assumptions have played an important role to push down the necessary regularity threshold for the well-posedness. Hidano \cite{Hid08} considered the mass-subcritical pure-power NLS:
\begin{equation}\label{NLS:purepower}
i\partial_{t} u - \Delta u =\eta |u|^{p-1}u \quad \textnormal{in} \ \R \times \R^d,
\end{equation}
where $\eta \in \C \setminus \{0\}$, $d \geq 3$, $4/(d+1)<p-1<4/d$, and established the global well-posedness of \eqref{NLS:purepower} in the scaling invariant space $\dot{H}^{\frac{d}{2}-\frac{2}{p-1}}(\R^d)$ for small and radial initial data. Since the scaling critical index is negative, radially symmetry assumption on initial data is essential for his result. Some negative results in this direction can be found in \cite{CCT03}. 
We refer to \cite{GW14} for the analogous global well-posedness results of 
\eqref{NLS:purepower}, and \cite{FW11}, \cite{CHO13} for the generalized results to the non-radial setting by assuming some angular regularity on data.

To state the main theorem, let us define the angular derivative and the function spaces. 
Let $x=(x^{(1)},x^{(2)}) \in \R^2$, and $\Delta_{\Sp^1} = (x^{(1)} \partial_{x^{(2)}} - x^{(2)} \partial_{x^{(1)}} )^2$ be the Laplace-Beltrami operator on $\Sp^1 \subset \R^2$. For $\sigma \in \R$, we define the angular derivative $\langle \Omega \rangle^{\sigma}$ by $(1-\Delta_{\Sp^1})^{\frac{\sigma}{2}}$. We give the precise definition of $\langle \Omega \rangle^{\sigma}$ in Section 2. 
Define the function space
\[
H^{s,\sigma}(\R^2) =\{ f \in \mathcal{S}'(\R^2) \, | \, \|f \|_{H^{s,\sigma}} := \|\langle \Omega \rangle^{\sigma} f \|_{H^s}< \infty\}.
\]
\begin{thm}\label{thm1}
Let $d=2$, $-1/2< s \leq -1/4$, and $\sigma \geq -2s -1/2$. Then \eqref{NLS} is locally well-posed in $H^{s,\sigma}(\R^2)$.
\end{thm}
\begin{rem}
We note that the local well-posedness of \eqref{NLS} in $H^{-\frac14}(\R^2) (=H^{-\frac{1}{4},0}(\R^2))$ was proved by Kishimoto \cite{Kishi08}. 
\end{rem}

We introduce the rough strategy and the key ingredients for the proof of Theorem \ref{thm1}. 
We will prove Theorem \ref{thm1} by the iteration argument which is a standard method to obtain the well-posedness. 
The candidate of solution spaces would be the Fourier restriction norm space $X^{s,b}$ which has been applied to the number of the studies of the dispersive and wave equations and provided numerous results. 
The point is that, however, if $\sigma = -2 s -1/2$, the $X^{s,b}$ space will not yield Theorem \ref{thm1}. The same difficulty arises in the one dimensional case. Kishimoto and Tsugawa \cite{KT10} employed the modified $X^{s,b}$ space to overcome the difficulty. We follow their approach and modify the $X^{s,b}$ space to handle the nonlinear interaction efficiently.  We explain the difficulty in further details in Section 2.

The second key ingredient in the proof of Theorem \ref{thm1} is the convolution estimate of the functions restricted to thickened hypersurfaces (see Assumption \ref{assumption:hypersurface} and Theorem \ref{theorem:trilinearest} below) which is the main contribution of our study. Angular regularity can be exploited effectively by using this estimate. 
Note that our approach for the proof of this estimate is inspired by the geometric observation by Bejenaru \cite{Bej06}. In Section 3, we introduce and discuss the estimate in more details. 

\vspace{3mm}
We show that Theorem~\ref{thm1} is in fact optimal.
\begin{thm}\label{thm2}
Let $d=2$, $s < -1/4$, and $0 \leq \sigma < \min(-2 s -1/2, 1/2)$. Then, for any $T>0$, the data-to-solution map $ u_0 \mapsto u$
of \eqref{NLS}, as a map from the unit ball in
$H^{s,\sigma}(\R^2)$ to
$C([0,T]; H^{s,\sigma}(\R^2))$ fails to be continuous at the origin.
\end{thm}
\begin{rem}
The condition $-1/2<s$ implies $- 2 s - 1/2 < 1/2$. Hence, Theorem~\ref{thm2} shows that Theorem~\ref{thm1} is an optimal result.
\end{rem}
In addition, we will see that even if an initial datum is radially symmetric, we cannot prove the well-posedness of \eqref{NLS} if $s \leq -1/2$ as long as we utilize the iteration argument. Let $H^s_{\mathrm{rad}}(\R^2)$ be the collection of radially symmetric functions belong to $H^s(\R^2)$.
\begin{thm}\label{thm3}
Let $d=2$ and $s \leq -1/2$. Then for any $T>0$, the data-to-solution map $ u_0 \mapsto u$
of \eqref{NLS}, as a map from the unit ball in
$H^s_{\mathrm{rad}}(\R^2)$ to
$C([0,T]; H^{s}_{\mathrm{rad}}(\R^2))$ fails to be $C^2$ at the origin.
\end{thm}
To see Theorem~\ref{thm3}, we will follow Bourgain's argument which is introduced in \cite{Bo97}. 

The paper is organized as follows. In Section 2, we introduce notations, explicit definition and some properties of $\langle \Omega \rangle^{\sigma}$. We introduce and discuss the function space where we will find the solution of \eqref{NLS}. In Section 3, we introduce the two key estimates which play an essential role in the proof of Theorem~\ref{thm1}. In Section 4, by using the estimates obtained in Section 3, we will establish the bilinear estimate which immediately produce Theorem~\ref{thm1}. The last section, Section 5, is devoted to the proof of the negative results, Theorems~\ref{thm2} and \ref{thm3}. 
The failure of the standard $X^{s,b}$ type bilinear estimate is also studied.
\section{Notation and Preliminaries}
We will use $A\lesssim B$ to denote an estimate of the form $A \le CB$ for some constant $C$ and write $A \sim B$ to mean $A \lesssim B$ and $B \lesssim A$. 
We will use the convention that capital letters denote dyadic numbers, e.g. $N=2^{n}$ for $n\in \N_0:=\N\cup\{0\}$ and for a dyadic summation we write $\sum_{N\geq M}a_{N}:=\sum_{n\in \N_0, 2^{n}\geq M}a_{2^{n}}$ for brevity. 
Let $\chi \in C^{\infty}_{0}((-2,2))$ be an even, non-negative function such that $\chi (t)=1$ for $|t|\leq 1$. 
We define $\psi (t):=\chi (t)-\chi (2t)$, 
$\psi_1(t):=\chi (t)$, and $\psi_{N}(t):=\psi (N^{-1}t)$ for $N\ge 2$. 
Then, $\sum_{N\geq 1}\psi_{N}(t)=1$.

Let $\mathbf{1}_A$ denotes the characteristic function of the set $A$.

We denote $\widehat{f}$ by the spatial Fourier transform of $f \in \mathcal{S}'(\R^2)$ and $\F_{t,x} u$ by the space and time Fourier transform of $u \in \mathcal{S}'(\R \times \R^2)$. Let $\F_{t,x}^{-1}$ be the space and time Fourier inverse transform. 
We define frequency Littlwood-Paley projections by $\widehat{P_{N}f}(\xi ):=\psi_{N}(\xi )\widehat{f}(\xi )$ 
and $e^{\pm it \Delta}$ by the propagators for the free Schr\"{o}dinger equations $i\partial_{t} u (t,x) \pm \Delta u(t,x) =0$.

We collect the fundamental properties of the angular derivative operator $\langle \Omega \rangle^{\sigma}$. 
In this paper, we only consider the two dimensional case. For the higher dimensional case, we refer to~\cite{Ster05, Ster07, CH18}.

First, for $f \in L^2(\R^2)$, we recall that the Fourier series expansion implies
\begin{align*}
f(x) & = f(|x| \cos \theta, |x| \sin \theta)\\
&  = \sum_{\ell \in \Z} 
\langle f(|x|\cos \theta, |x| \sin \theta), e^{i \ell \theta} \rangle_{L^2_{\theta}} e^{i \ell \theta} \ \ \textnormal{in $L^2(\R^2)$},
\end{align*}
where $\langle \omega_1, \omega_2 \rangle_{L_{\theta}^2} = (2 \pi)^{-1}\int_0^{2 \pi} \omega_1(\theta) \overline{\omega_2}(\theta) d \theta$. 
The operator $\langle \Omega \rangle^{\sigma}$ is formulated as 
\[
\langle \Omega \rangle^{\sigma}f(x) = \sum_{\ell \in \Z} \langle \ell \rangle^{\sigma}
\langle f(|x|\cos \theta, |x| \sin \theta), e^{i \ell \theta} \rangle_{L^2_{\theta}} e^{i \ell \theta}.
\]
Let $M \in 2^{\N_0}$. We define the spherical Littilewood-Paley projections
\begin{equation}\label{definition:H_M}
\begin{split}
(H_M f)(x) &= (H_M f)(|x| \cos \theta, |x|\sin \theta)\\
& = \sum_{\ell \in \Z}\psi_M(\ell) \langle f(|x|\cos \theta, |x| \sin \theta), e^{i \ell \theta} \rangle_{L^2_{\theta}} e^{i \ell \theta}.
\end{split}
\end{equation}
Then, for $f \in L^2(\R^2)$, it holds that
\[
f(x) = \sum_{M \geq 1} H_M f(x) \ \ \textnormal{in $L^2(\R^2)$}, \quad \|\langle \Omega \rangle^{\sigma} (H_M f)\|_{L^2} \sim M^{\sigma} \|H_M f\|_{L^2},
\]
and thus, for $s$, $\sigma \in \R$, we may write 
\[
\|f \|_{H^{s,\sigma}} \sim \bigl( \sum_{N \geq 1,M \geq 1} N^{2s} M^{2 \sigma} \|  H_M P_N f \|_{L_x^2}^2 \bigr)^{\frac12}.
\]
It is well-known that if $f \in L^2(\R^2)$ can be described as $f(|x|\cos \theta, |x| \sin \theta) = f_0(|x|)e^{i \ell \theta}$ then there exists $F_0$ such that 
$\widehat{f}(|\xi| \cos \phi, |\xi| \sin \phi) = F_0(|\xi|)e^{i \ell \phi}$ (see e.g. \cite[Chapter 4]{SW71}). 
This implies that 
$\langle \Omega \rangle^{\sigma} \F_{x} = \F_{x} \langle \Omega \rangle^{\sigma}$ 
and, in particular, it holds that $\|\langle \Omega \rangle^{\sigma} \widehat{H_M f} \|_{L^2} \sim M^{\sigma} \| H_M f\|_{L^2}$. 
We will frequently use this relation throughout the paper.
\subsection{Function Spaces}\label{subsection2.1}
We introduce the function spaces. 
\begin{defn}\label{definition2.1}
For $L \in 2^{\N_0}$ and $u \in \mathcal{S}' (\R \times \R^2)$, we define the operator $Q_L$ by $\F_{t,x}{Q_{L}u}(\tau ,\xi ):=\psi_{L}(\tau - |\xi|^{2})\F_{t,x}{u}(\tau ,\xi )$. 
Let $s\in \R$, $b\in \R$, and $\sigma \geq 0$. Define
\begin{align*}
& X^{s,\sigma;b}(\R \times \R^2):=\{u\in \mathcal{S}' (\R \times \R^2) \, | \, \|u\|_{X^{s,\sigma;b}}<\infty\},\\
& \|u\|_{X^{s,\sigma;b}}:= \Bigl( \sum_{N\geq 1} \sum_{M \geq 1}
\sum_{L\geq 1}N^{2s} M^{2 \sigma} L^{2b}\|Q_{L}H_M P_{N}u\|_{L^2}^2\Bigr)^{\frac{1}{2}},\\
& Y^{s,\sigma}(\R \times \R^2):=\{u\in \mathcal{S}' (\R \times \R^2) \, | \, \|u\|_{Y^{s,\sigma} }<\infty\},\\
& \|u\|_{Y^{s,\sigma} } := \begin{cases}
\Bigl( \sum_{N\geq 1} \sum_{M \geq 1}N^{2s} M^{2 \sigma} \|\F_{t,x}H_M P_{N}u\|_{L_{\xi}^2 L_{\tau}^1}^2\Bigr)^{\frac{1}{2}} \ & (\sigma >0),\\
\Bigl( \sum_{N\geq 1}N^{2s}\|\F_{t,x} P_{N}u\|_{L_{\xi}^2 L_{\tau}^1}^2\Bigr)^{\frac{1}{2}} \ &(\sigma =0).
\end{cases}
\end{align*}
\end{defn}
\begin{rem}
We give a comment about Definition~\ref{definition2.1} in the case $\sigma=0$. It is easy to see
\[
\|u\|_{X^{s,0;b}}= \Bigl( \sum_{N\geq 1}
\sum_{L\geq 1}N^{2s}  L^{2b}\|Q_{L}H_M P_{N}u\|_{L^2}^2\Bigr)^{\frac{1}{2}}.
\]
While, since technical issues arise in the proof of the nonlinear estimate, e.g. in the proof of \eqref{est:prop4.1-15} below, the definition of the $Y^{s,0}$-norm is given in a slight different way from $Y^{s,\sigma}$ with $\sigma >0$. 
\end{rem}
\begin{defn}\label{definition2.15}
Let $s\in \R$, $\sigma \geq 0$, and $\kappa>0$. Define
\begin{align*}
& Z^{s,\sigma}_{\kappa} = X^{s,\sigma;\frac12 + \kappa} +(X^{s+4 \kappa, \sigma;\frac12 - \kappa} \cap Y^{s,\sigma}),\\
& \|u\|_{Z^{s,\sigma}_{\kappa}} = \inf_{u = u_1+u_2} \bigl( \|u_1\|_{X^{s,\sigma;\frac12 + \kappa}} + \|u_2\|_{X^{s + 4 \kappa, \sigma;\frac12 - \kappa}} + \|u_2\|_{Y^{s,\sigma}} \bigr).
\end{align*}
For a time interval $I \subset \R$, we define the time restricted space $Z^{s,\sigma}_{\kappa}(I)$ as follows:
\begin{align*}
& Z^{s,\sigma}_{\kappa}(I):=\{u\in C(I ; H^{s,\sigma}(\R^2)) \, | \, \|u\|_{Z^{s,\sigma}_{\kappa}(I)}<\infty\},\\
& \|u\|_{Z^{s,\sigma}_{\kappa}(I)} := \inf \{ \|v\|_{Z^{s,\sigma}_{\kappa} }\, | \, u (t) = v (t) \ \mathrm{if} \ t \in I \}.
\end{align*}
\end{defn}
\begin{rem}
\begin{enumerate}[label=(\arabic*),  leftmargin=20pt]
\item  It is observed that $\|u\|_{Y^{s,\sigma}} \lesssim \|u\|_{X^{s, \sigma;\frac12 +\kappa}}$, and if $L \sim N^2$ then $\| Q_{L}P_N u\|_{X^{s, \sigma;\frac12 +\kappa}} \sim \|Q_{L} P_Nu\|_{X^{s + 4 \kappa, \sigma;\frac12 - \kappa}}$. 
Thus, it holds that
\begin{equation}\label{est:Znorm}
\begin{split}
&\|P_N u\|_{Z^{s,\sigma}_{\kappa}}\\
& \sim \|Q_{L \ll N^2}P_N u\|_{X^{s, \sigma;\frac12 +\kappa}} + \|Q_{L \gtrsim N^2} P_Nu\|_{X^{s + 4 \kappa, \sigma;\frac12 - \kappa}} + \|Q_{L \gtrsim N^2} P_N u\|_{Y^{s,\sigma}}.
\end{split}
\end{equation}
\item For $u \in C(\R;L^2(\R^2))$ such that $\F_{t,x}u \in L_{\xi}^2 L_{\tau}^1$, it is observed that
\[
\| u(t)\|_{L_x^2} = \Bigl\| \int e^{i t \tau} (\F_{t,x} u)(\tau,\xi) d \tau \Bigl\|_{L_{\xi}^2} \leq \|\F_{t,x} u \|_{L_{\xi}^2 L_{\tau}^1}.
\]
This and $\|u\|_{Y^{s,\sigma}} \lesssim \|u\|_{X^{s, \sigma;\frac12 +\kappa}}$ imply that, for $u \in Z^{s,\sigma}_{\kappa}(I)$, it holds that
\begin{equation}
\label{est:embeddingZ}
\| u\|_{C(I; H^{s,\sigma})} \lesssim \|u\|_{Z^{s,\sigma}_{\kappa}(I)}.
\end{equation}
Hence, $Z^{s,\sigma}_{\kappa}(I)$ is a Banach space.
\end{enumerate}
\end{rem}

We comment on why we need to employ the function space $Z^{s,\sigma}_{\kappa}$ instead of simply using $X^{s,\sigma;b}$. 
In \cite{NTT01}, Nakanishi, Takaoka, and Tsutsumi proved that if $b \geq 1/2$, the following bilinear estimate fails to hold:
\begin{equation}\label{est:NTT}
\|u \overline{v}\|_{X^{-\frac14;b-1}} \lesssim \|u\|_{X^{-\frac14;b}}\|v\|_{X^{-\frac14;b}},
\end{equation}
where
\begin{align*}
& X^{s;b}(\R \times \R):=\{u\in \mathcal{S}' (\R \times \R) \, | \, \|u\|_{X^{s;b}}<\infty\},\\
& \|u\|_{X^{s;b}}:= \Bigl( \sum_{N\geq 1} 
\sum_{L\geq 1}N^{2s}  L^{2b}\|Q_{L} P_{N}u\|_{L^2}^2\Bigr)^{\frac{1}{2}}.
\end{align*}
We notice that the estimate \eqref{est:NTT} is required to show the well-posedness of \eqref{NLS} with $d=1$ in $H^{-\frac14}(\R)$ by using the contraction mapping argument in $X^{-\frac14; b}$ with $b > 1/2$. 
Hence, the failure of \eqref{est:NTT} means that the standard $X^{s;b}$ space is insufficient for the well-posedness of \eqref{NLS} in $H^{-\frac14}(\R)$. 
To overcome this, Kishimoto and Tsugawa \cite{KT10} employed the modified Fourier restriction norm and established the well-posedness result in $H^{-\frac14}(\R)$. Note that their strategy, modifying the Fourier restriction norm to handle dangerous nonlinear interactions, was inspired by the work by Bejenaru and Tao \cite{BT06}. 
Additionally, in the two dimensional case, Kishimoto \cite{Kishi08} employed the modified Fourier restriction norm to prove the well-posedness of \eqref{NLS} in $H^{-\frac14}(\R^2)$. 

By following the proof of Theorem~1 (iii) in~\cite{NTT01} by Nakanishi, Takaoka, and Tsutsumi, the failure of \eqref{est:NTT} can be extended to the $\R \times \R^2$ case, that is, if $b \geq 1/2$, the following estimate fails to hold:
\[
\|u \overline{v}\|_{X^{-\frac14,0;b-1}} \lesssim \|u\|_{X^{-\frac14,0;b}}\|v\|_{X^{-\frac14,0;b}}.
\]
Here, the $X^{-\frac14,0;b}$-norm is the same as in Definition \ref{definition2.1}. 
Furthermore, we may generalize this fact as follows:
\begin{prop}\label{proposition2.2}
Let $-1/2< s \leq -1/4$, and $b \geq 1/2$. Then, the following estimate fails to hold:
\[
\|u \overline{v}\|_{X^{s,- 2 s -\frac12;b-1}} \lesssim \|u\|_{X^{s,-2 s -\frac12;b}}\|v\|_{X^{s,-2 s -\frac12;b}}.
\]
\end{prop}
We will give the proof of Proposition~\ref{proposition2.2} in Section~\ref{Section5}. 
\begin{rem}
Proposition~\ref{proposition2.2} suggests that we cannot show the well-posedness of \eqref{NLS} in $H^{s, -2s -\frac12}$, which is the critical case in Theorem \ref{thm1}, by just using the $X^{s,- 2 s -1/2;b}$ space. 
Thus, to prove Theorem \ref{thm1}, we follow the approach by Kishimoto and Tsugawa \cite{KT10} and employ the modified Fourier restriction norm $Z^{s,\sigma}_{\kappa}$. 
While, in the case $\sigma > -2 s - 1/2$, the modified space $Z^{s,\sigma}_{\kappa}$ is not required and the standard $X^{s,\sigma;b}$ space yields Theorem \ref{thm1}. That is, if $-1/2< s \leq -1/4$ and $\sigma >-2 s -1/2$ then there exists $b > 1/2$ such that the following holds:
\[
\|u \overline{v}\|_{X^{s,\sigma;b-1}} \lesssim \|u\|_{X^{s,\sigma;b}}\|v\|_{X^{s,\sigma;b}}.
\]
\end{rem}
\subsection{A Key bilinear estimate and proof of Theorem~\ref{thm1}}\label{subsection2.2}
We introduce a key bilinear estimate which immediately establishes Theorem \ref{thm1} by the iteration argument. 
\begin{prop}\label{prop2.2}
Let $-1/2< s \leq -1/4$, and $\sigma \geq -2s -1/2$. Then there exists small $\kappa=\kappa(s)>0$ such that
\[
\|\F_{t,x}^{-1} \bigl( \langle \tau - |\xi|^2 \rangle^{-1} \F_{t,x}({u}_1{\overline{u_2}})\bigr)\|_{Z^{s,\sigma}_{\kappa}} \lesssim 
\|u_1\|_{Z^{s,\sigma}_{\kappa}} \|u_2\|_{Z^{s,\sigma}_{\kappa}}.
\]
\end{prop}
Lastly in this section, we see that Proposition~\ref{prop2.2} yields Theorem~\ref{thm1}. Since the proof is standard, we give only a rough sketch of the proof. For the complete proof of Theorem \ref{thm1}, we refer to \cite{GTV97}, \cite{BT06}, \cite{Kishi09}, \cite{KT10}. 

Since \eqref{NLS} with $d=2$ is scaling invariant in $\dot{H}^{-\frac12}(\R^2)$, for the proof of Theorem~\ref{thm1}, we may assume that $\|u_0\|_{H^{s,\sigma}}$ is sufficiently small. 

We define
\begin{equation}\label{definition:N}
\mathcal{N}(u_1,u_2)(t) = - i \eta \int_0^t e^{-i(t-t') \Delta} \bigl( u_1(t') \overline{u_2}(t')\bigr) d t'.
\end{equation}
Let $I=[0,1]$. 
Our goal is to prove that the following operator
\begin{align*}
\mathcal{K}[u](t)&=e^{-it \Delta} u_0 - i \eta \int_0^t e^{-i(t-t') \Delta} |u(t')|^2 d t'\\
& =e^{-it \Delta} u_0 + \mathcal{N}(u,u)(t)
\end{align*}
is a contraction mapping in $Z^{s,\sigma}_{\kappa}(I)$. 
First, let us introduce the linear estimate. 
\begin{lem}\label{lemma2.4}
Let $s \in \R$, $\sigma \geq 0$, $0<\kappa<1$, and $I=[0,1]$. Then, we have
\[
\|e^{-it \Delta} u_0\|_{Z^{s,\sigma}_{\kappa}(I)}  \lesssim \|u_0\|_{H^{s,\sigma}}.
\]
\end{lem}
The proof of Lemma~\ref{lemma2.4} is simple since $\|e^{-it \Delta} u_0\|_{Z^{s,\sigma}_{\kappa}(I)}  \lesssim \|\chi(t) e^{-it \Delta} u_0\|_{X^{s,\sigma;1}}$ where $\chi$ is defined at the beginning of Section~2. 
Thus we omit the proof. For the details, see e.g. \cite{GTV97}, \cite[Proposition~2.6]{KT10}.

Next, we introduce the estimate to handle the Duhamel term $\mathcal{N}(u,u)$. 
We omit the proof. 
See \cite[Lemma~2.1]{GTV97} and \cite[Proposition~2.7]{KT10}.
\begin{lem}\label{lemma2.5}
Let $s \in \R$, $\sigma \geq 0$, $0<\kappa<1$, and $v_1$, $v_2 \in Z^{s,\sigma}_{\kappa}$. Then, we have
\[
\|\chi (t) \mathcal{N}(v_1,v_2)\|_{Z^{s,\sigma}_{\kappa}}  \lesssim \|\F_{t,x}^{-1} \bigl( \langle \tau - |\xi|^2 \rangle^{-1} \F_{t,x}(v_1 \overline{v_2})\bigr)\|_{Z^{s,\sigma}_{\kappa}}.
\]
\end{lem}
Let $I=[0,1]$. Lemma~\ref{lemma2.5} and Proposition~\ref{prop2.2} imply that if $-1/2< s \leq -1/4$ and $\sigma \geq -2s -1/2$ then there exists $\kappa=\kappa(s)>0$ such that
\begin{equation}\label{est:NonlinearEst}
\|\mathcal{N}(u,u)\|_{Z^{s,\sigma}_{\kappa}(I)}  \lesssim \|u\|_{Z^{s,\sigma}_{\kappa}(I)}^2.
\end{equation}
To see this, we choose $v \in Z^{s,\sigma}_{\kappa}$ which satisfies
\[
u(t) = v(t) \quad \mathrm{if} \ t \in I, \qquad \|v \|_{Z^{s,\sigma}_{\kappa}} \leq 2 \|u\|_{Z^{s,\sigma}_{\kappa}(I)}.
\]
It is known that if $v \in Z^{s,\sigma}_{\kappa}$ then $\mathcal{N}(v,v) \in C(\R; H^{s,\sigma}(\R^2))$. 
For the proof, we refer to Lemma~2.2 in \cite{GTV97}. 
Then, since $\mathcal{N}(v,v)(t) = \mathcal{N}(u,u)(t)$ if $t \in I$, it holds that $\|\mathcal{N}(u,u)\|_{Z^{s,\sigma}_{\kappa}(I)} \leq \|\chi(t) \mathcal{N}(v,v)\|_{Z^{s,\sigma}_{\kappa}}$. 
Hence, it follows from Lemma~\ref{lemma2.5} and Proposition~\ref{prop2.2} that
\begin{align*}
\|\mathcal{N}(u,u)\|_{Z^{s,\sigma}_{\kappa}(I)} & \leq \|\chi(t) \mathcal{N}(v,v)\|_{Z^{s,\sigma}_{\kappa}}\\
&\lesssim \|\F_{t,x}^{-1} \bigl( \langle \tau - |\xi|^2 \rangle^{-1} \F_{t,x}(|v|^2)\bigr)\|_{Z^{s,\sigma}_{\kappa}}\\
& \lesssim  \|v \|_{Z^{s,\sigma}_{\kappa}}^2 \lesssim \|u\|_{Z^{s,\sigma}_{\kappa}(I)}^2,
\end{align*}
which completes the proof of \eqref{est:NonlinearEst}.

Since $\|u_0\|_{H^{s,\sigma}}$ is small, by the standard argument, we may see that the mapping $\mathcal{K}$ is contraction in $Z^{s,\sigma}_{\kappa}(I)$ and then the existence of the solution is verified. In addition, the uniqueness and the continuous dependence on initial data can be shown in a standard way. 

\section{Key convolution estimates}
In this section, we introduce convolution estimates which play a central role in the proof of Proposition \ref{prop2.2}. 
The fist estimate is the convolution estimate on the hypersurfaces which can be seen as a nonlinear version of the Loomis-Whitney inequality. First, we define $C^{1,\beta}$-hypersurfaces given as graphs of $C^{1,\beta}$-functions following \cite[Assumption~1.1]{BH11} and \cite{KS21}.
\begin{defn}\label{Definition:Hypersurfaces}
For $0 < \beta \leq 1$, we say that $S \subset \R^3$ is a $C^{1,\beta}$ hypersurface if there exist an open and convex set $\Omega \subset \R^2$ and a $C^{1,\beta}$ function $\varphi : \Omega \to \R$ such that
\[
S = \{ (\varphi(\xi), \xi) \in \R^3 \; | \; \xi \in \Omega \}.
\]
\end{defn}
\begin{rem}
Let $S$ be a $C^{1,\beta}$ hypersurface and $\mathfrak{n}$ denote the unit normal vector field on $S$. 
Then, $\mathfrak{n}$ satisfies the H\"older condition
\[
\sup_{\sigma, \tilde{\sigma} \in S} \biggl( \frac{|\mathfrak{n}(\sigma) - \mathfrak{n}(\tilde{\sigma})|}{|\sigma - \tilde{\sigma}|^\beta} + \frac{|\mathfrak{n}(\sigma) \cdot (\sigma - \tilde{\sigma})|}{|\sigma - \tilde{\sigma}|^{1+\beta
}}\biggr) < \infty.
\]
\end{rem}
\begin{assumption}\label{AssumptionSurfaces}
Let $S_1$, $S_2$, $S_3$ be $C^{1,\beta}$ hypersurfaces. For $i=1,2,3$, let $\mathfrak{n}_i$ denotes the unit normal vector field on $S_i$. 
There exists $A \geq 1$ such that the matrix $N(\sigma_1,\sigma_2,\sigma_3) = (\mathfrak{n}_1(\sigma_1),\mathfrak{n}_2(\sigma_2),\mathfrak{n}_3(\sigma_3))$ satisfies the transversality condition
\begin{equation}
\label{eq:TransversalityAssumption}
A^{-1} \leq |\det N(\sigma_1,\sigma_2,\sigma_3)| \leq 1
\end{equation}
for all $(\sigma_1,\sigma_2,\sigma_3) \in S_1 \times S_2 \times S_3$.
\end{assumption}
If hypersurfaces $S_1$, $S_2$, $S_3$ satisfy Assumption~\ref{AssumptionSurfaces}, the following convolution estimate of functions restricted to such hypersurfaces holds. 
\begin{thm}
\label{thm:GNLWR3}
Assume that $S_1$, $S_2$, $S_3$ satisfy Assumption~\ref{AssumptionSurfaces}. Then, for each $f \in L^2(S_1)$ and $g \in L^2(S_2)$, we have
\begin{equation}\label{est:GNLWR3}
\Vert f * g \Vert_{L^2(S_3)} \lesssim A^{\frac12} \Vert f \Vert_{L^2(S_1)} \Vert g \Vert_{L^2(S_2)},
\end{equation}
where the implicit constant is independent of $\beta$.
\end{thm}
For the proof of Theorem~\ref{thm:GNLWR3}, we refer to Theorem 4.1 in \cite{KS21}. 
The convolution estimate \eqref{est:GNLWR3}, for functions restricted to $C^3$-regularity localized hypersurfaces, was introduced by Bennett, Carbery, and Wright in \cite{BCW05} as a perturbed version of the classical three dimensional Loomis-Whitney inequality \cite{LW}. After that, the regularity assumption on hypersurfaces was relaxed, and \eqref{est:GNLWR3} for $C^{1,\beta}$ ($\beta >0$) localized hypersurfaces was obtained in \cite{BHHT09, BHT10}. See also~\cite{BH11}. 
Lastly, the locality assumption was removed in \cite{KS05, KS21}.

The estimate \eqref{est:GNLWR3} has a connection with the trilinear restriction estimates (see e.g. \cite{BCW05, BCT06}) and has been applied to the study of the Cauchy problem of nonlinear dispersive equations. 
For instance, \eqref{est:GNLWR3} was utilized in \cite{BHHT09} to push down the necessary regularity threshold for the well-posedness of the Zakharov system in $\R^2$. 

For $\varepsilon>0$ and a $C^{1,\beta}$ hypersurface $S$ given by $\varphi : \Omega \to \R$, the thickened hypersurface $S(\varepsilon)$ is defined by
\begin{equation}\label{defn:ThickenedS}
S(\varepsilon) = \{ (\tau,\xi) \in \R^3 \, |\, \xi \in \Omega, \ |\tau-\varphi(\xi)| \leq \varepsilon \}.
\end{equation}
We note that in the proof of the key bilinear estimate (Proposition~\ref{prop2.2}), instead of Theorem~\ref{thm:GNLWR3}, we use the following convolution estimate on thickened hypersurfaces:
\begin{equation}\label{est:ThickenedLW}
\Vert f * g \Vert_{L^2(S_3(\varepsilon))} \lesssim \varepsilon^{\frac32} A^{\frac12} \Vert f \Vert_{L^2(S_1(\varepsilon))} \Vert g \Vert_{L^2(S_2(\varepsilon))},
\end{equation}
where $C^{1,\beta}$ hypersurfaces $S_1$, $S_2$, $S_3$ satisfy \textit{Assumption}~\ref{AssumptionSurfaces}. 
It is straightforward to see that Theorem~\ref{thm:GNLWR3}, Fubini's theorem, and the Cauchy-Schwarz inequality give \eqref{est:ThickenedLW}, and in fact, we may see that the two estimates \eqref{est:GNLWR3} and \eqref{est:ThickenedLW} are equivalent. For the details, we refer to Section 4.2. in \cite{KS21}. 

Next, we state the second key estimate which is similar in form to \eqref{est:ThickenedLW}. For $2 \leq p< \infty$, we define
\begin{align*}
& \|f\|_{L_{\tau,r}^2 L_{\theta}^{p}(S_i(\varepsilon))} = \Bigl( \int_{\R} \int_0^{\infty}  \Bigl(\int_0^{2 \pi}  \mathbf{1}_{S_i(\varepsilon)} |f(\tau, r \cos \theta, r \sin \theta)|^p d\theta\Bigr)^{\frac{2}{p}} r  dr d\tau \Bigr)^{\frac12},\\
& \|f\|_{L_{\tau,r}^2 L_{\theta}^{\infty}(S_i(\varepsilon))} = \Bigl( \int_{\R} \int_0^{\infty}  \sup_{\theta \in [0,2 \pi)} \mathbf{1}_{S_i(\varepsilon)} |f(\tau, r \cos \theta, r \sin \theta)|^2 r  dr d\tau \Bigr)^{\frac12}.
\end{align*}
Let $B_2 = \{ \xi \in \R^2 \, | \, |\xi| < 2\}$. The assumptions on hypersurfaces are the following:
\begin{assumption}\label{assumption:hypersurface}
Suppose that the $C^{1,\beta}$ hypersurfaces $S_1$, $S_2$, $S_3$ are given by $\varphi_1$, $\varphi_2$, $\varphi_3 \in C^{1,\beta}({B_2})$, respectively, which satisfy the following conditions:
\begin{enumerate}[label=(\roman*)]
\item \label{assumption1-Ass2} For $i=1,2,3$, $\varphi_i$ is radial,
\item \label{assumption2-Ass2} $|\nabla \varphi_i(\xi)|\sim 1 \quad \textnormal{if} \ \  |\xi| \sim 1$,
\item \label{assumption3-Ass2} If $\xi_1+\xi_2+\xi_3=0$, $\max_{i=1,2,3}|\xi_i| \sim 1$, and 
$|\varphi_1(\xi_1)+ \varphi_2(\xi_2) + \varphi_3(\xi_3)| \ll 1$, then it holds that 
\[
| \nabla \varphi_1(\xi_1)- \nabla \varphi_2 (\xi_2)|+|\nabla \varphi_2(\xi_2)- \nabla \varphi_3 (\xi_3)| \sim 1,
\]
\end{enumerate}
\end{assumption}

\begin{thm}\label{theorem:trilinearest}
For $i=1,2,3$, let $0< \varepsilon$, $N_i \in 2^{\Z}$, and a function $f_i \in L^2(S_i(\varepsilon))$ satisfies 
$\supp f_i \subset \{(\tau,\xi) \, | \, |\xi| \sim N_i\}$. Assume that $C^{1,\beta}$ hypersurfaces $S_1$, $S_2$, $S_3$ satisfy \textit{Assumption} \textnormal{\ref{assumption:hypersurface}} and $N_3 \lesssim N_1 \sim N_2 \sim 1$. 
Then, for $2 \leq p \leq \infty$, we have
\begin{equation}\label{est:ThickenedLWrad}
| (f_1 * f_2 * f_3)(0)| \lesssim C(\varepsilon,p, N_3)  \|f_1\|_{L_{\tau,r}^2 L_{\theta}^{p}(S_1(\varepsilon))} \|f_2\|_{L^2(S_2(\varepsilon))} \|f_3\|_{L^2(S_3(\varepsilon))},
\end{equation}
where 
\begin{equation}
\label{defn:ConstantThm3.4}
C(\varepsilon,p, N_3) =
\begin{cases}
\varepsilon^{\frac32 - \frac{1}{p}} N_{3}^{\frac{1}{p}} \ & (2 \leq p < \infty),\\
\langle \log \varepsilon \rangle \varepsilon^{\frac32} \ & (p= \infty).
\end{cases}
\end{equation}
\end{thm}
\begin{rem}
In contrast to Theorem~\ref{thm:GNLWR3}, the implicit constant of \eqref{est:ThickenedLWrad} depends on $\beta$.
\end{rem}
The following proposition shows that Theorem~\ref{theorem:trilinearest} is almost optimal.
\begin{prop}\label{prop:SharpnessFactor}
In the estimate \eqref{est:ThickenedLWrad}, the factor $C(\varepsilon,p,N_3)$ defined in \eqref{defn:ConstantThm3.4} is optimal if $2 \leq p < \infty$ and almost optimal if $p=\infty$. Specifically, under the same assumptions for $S_i$, $N_i$ $(i=1,2,3)$ as in Theorem~\ref{theorem:trilinearest}, for any $0< \varepsilon \ll 1$, there exist $f_1$, $f_2$, $f_3$ satisfying $\supp f_i \subset \{(\tau,\xi) \, | \, |\xi| \sim N_i\}$ such that
\begin{equation}\label{est:ThickenedLWrad2}
| (f_1 * f_2 * f_3)(0)| \gtrsim \widetilde{C}(\varepsilon,p, N_3)  \|f_1\|_{L_{\tau,r}^2 L_{\theta}^{p}(S_1(\varepsilon))} \|f_2\|_{L^2(S_2(\varepsilon))} \|f_3\|_{L^2(S_3(\varepsilon))},
\end{equation}
where 
\[
\widetilde{C}(\varepsilon,p, N_3) =
\begin{cases}
\varepsilon^{\frac32 - \frac{1}{p}} N_{3}^{\frac{1}{p}} \ & (2 \leq p < \infty),\\
\langle \log \varepsilon \rangle^{\frac12} \varepsilon^{\frac32} \ & (p= \infty).
\end{cases}
\]
\end{prop}
In Section~\ref{Section5}, we will prove Proposition~\ref{prop:SharpnessFactor} by giving explicit functions $f_1$, $f_2$, $f_3$ which satisfy \eqref{est:ThickenedLWrad2}.

Note that, by duality, it holds that
\[
\sup_{\|f_3 \|_{L^2(S_3(\varepsilon))} =1} | (f_1 * f_2 * f_3)(0)| = \|f_1 * f_2 \|_{L^2(S_3^- (\varepsilon))}, 
\]
where $S_3^- = \{ (\tau, \xi) \in \R^3 \,| \, (-\tau,-\xi) \in S_3 \}$. Thus \eqref{est:ThickenedLWrad} is equivalent to
\[
\|f_1 * f_2 \|_{L^2(S_3^- (\varepsilon))} \lesssim C(\varepsilon,p, N_3)  \|f_1\|_{L_{\tau,r}^2 L_{\theta}^{p}(S_1(\varepsilon))} \|f_2\|_{L^2(S_2(\varepsilon))},
\]
which is similar in form to \eqref{est:ThickenedLW}.

Before turning to the proof, we give comments on \textit{Assumption} \ref{assumption:hypersurface} and Theorem \ref{theorem:trilinearest}. 
\begin{rem}\label{rem3.3}
\begin{enumerate}[label=(\arabic*),  leftmargin=20pt]
\item Hypersurfaces which satisfy \textit{Assumption} \ref{assumption:hypersurface} are localized in space around origin and given as graphs of radially symmetric functions. We need such spatial radially symmetry assumption on hypersurfaces to exploit the angular regularity assumption of the function $f_1$.
\item The key difference between Theorems \ref{thm:GNLWR3} and \ref{theorem:trilinearest} is that in \textit{Assumption} \ref{assumption:hypersurface}, the transversality condition on hypersurfaces \eqref{eq:TransversalityAssumption} in \textit{Assumption} \ref{AssumptionSurfaces} is not required. We will see that Theorem \ref{theorem:trilinearest} provides the better estimate than that given by Theorem \ref{thm:GNLWR3} by exploiting that angular integrability condition on $f_1$.
\item The assumption \ref{assumption3-Ass2} in \textit{Assumption} \ref{assumption:hypersurface} is related to the space-time resonance in \cite{GMS09}. Precisely, the assumption \ref{assumption3-Ass2} ensures that the space-time resonance set
\[
\biggl\{ (\xi_1,\xi_2) \in \R^2 \times \R^2 \, \biggl| \, 
\begin{aligned} &|\varphi_1(\xi_1)+ \varphi_2(\xi_2) + \varphi_3(-\xi_1-\xi_2)|=0, \\ 
&| \nabla \varphi_1(\xi_1)- \nabla \varphi_2 (\xi_2)| =0.
\end{aligned}
\biggr\}
\]
is separated from $\{(\xi_1,\xi_2) \, | \, |\nabla \varphi_2(\xi_2)- \nabla \varphi_3 (- \xi_1-\xi_2)|=0 \}$.
\item Our approach for the proof of Theorem \ref{theorem:trilinearest} is inspired heavily by the geometric argument in the proof of Proposition 5.6 in \cite{Bej06}. 
\item We may expect that Theorem \ref{theorem:trilinearest} will be applied to the study of variant dispersive equations. For example, we most likely obtain the analogous nontrivial result for the generalized equation:
\[
i\partial_{t} u (t,x) + (-\Delta)^{\frac{a}{2}} u(t,x) =\eta |u(t,x)|^2,
\] 
with $a>1$ under some angular regularity assumption. See examples of hypersurfaces below which satisfy Assumption \ref{assumption:hypersurface}. For this reason, in \textit{Assumption} \ref{assumption:hypersurface}, the necessary conditions of hypersurfaces are given in a general way, 
which might seem unnecessarily complex to show Theorem \ref{thm1}. 
\end{enumerate}
\end{rem}
Here are a few examples of hypersurfaces which satisfy \textit{Assumption} \ref{assumption:hypersurface}.
\begin{rem}\label{rem3.4}
\begin{enumerate}[label=(\arabic*),  leftmargin=20pt]
\item Let $\varphi_{a,\pm}(\xi)=\pm |\xi|^a$ with $a>1$. For $j=1,2,3$, hypersurfaces given by the functions $\varphi_j = \varphi_{a,\pm_j}$ satisfy \textit{Assumption} \ref{assumption:hypersurface} regardless of the choice of the signs $\{\pm_1,\pm_2,\pm_3\}$. 
Hence, in particular, the hypersurfaces given by $\varphi_{2,\pm_j}(\xi)=\pm_j |\xi|^2$, 
the phase functions of the free Schr\"{o}dinger equation $i \partial_t u \mp_j \Delta u=0$, satisfy \textit{Assumption}~\ref{assumption:hypersurface}, and we can utilize Theorem \ref{theorem:trilinearest} in the proof of Proposition~\ref{prop2.2}.
\item Let $\sigma_1$, $\sigma_2$, $\sigma_3 \in \R \setminus \{0\}$. Assume $1/\sigma_1 + 1/\sigma_2 + 1 / \sigma_3 \not= 0$. Then, the hypersurfaces given by $\varphi_1 (\xi)= \sigma_1 | \xi|^2$, $\varphi_2 (\xi)= \sigma_2|\xi|^2$, $\varphi_3 (\xi)= \sigma_3 |\xi|^2$ satisfy \textit{Assumption}~\ref{assumption:hypersurface}. 
While, if $1/\sigma_1 + 1/\sigma_2 + 1 / \sigma_3 = 0$, the given hypersurfaces satisfy the first two assumptions of 
\textit{Assumption}~\ref{assumption:hypersurface} but fail to fulfill the last one. 
Indeed, for any $c \in \R$, taking $\xi_1 = (\sigma_2 c, 0)$, $\xi_2 = (\sigma_1 c, 0)$, $\xi_3=( - (\sigma_1 + \sigma_2) c,0)$, we may see that $\xi_1+\xi_2+\xi_3=0$, $\varphi_1(\xi_1) + \varphi_2(\xi_2) + \varphi_3(\xi_3)=0$, and
\[
| \nabla \varphi_1(\xi_1)- \nabla \varphi_2 (\xi_2)|+|\nabla \varphi_2(\xi_2)- \nabla \varphi_3 (\xi_3)|=0.
\]
If $1/\sigma_1 + 1/\sigma_2 + 1 / \sigma_3 = 0$, for the hypersufaces $S_1$, $S_2$, $S_3$ given by 
$\varphi_1$, $\varphi_2$, $\varphi_3$, we will see that \eqref{est:ThickenedLWrad} fails to hold. 
We can take $\sigma_1$, $\sigma_2$, $\sigma_3 \in \R \setminus \{0\}$ so that $1/\sigma_1 + 1/\sigma_2 + 1 / \sigma_3 = 0$ and $\sigma_1+\sigma_2 \not= 0$. 
Then, for any $0<\varepsilon \ll 1$, if we define $f_1$, $f_2$, $f_3$ by
\[
f_1 = \mathbf{1}_{C_1}, \quad f_2 = \mathbf{1}_{C_2}, \quad f_3 = \mathbf{1}_{C_3},
\]
where
\begin{align*}
& C_1 = \{ (\tau,\xi) \in \R^3 \, |\, |\tau-\sigma_1|\xi|^2|\leq \varepsilon, \quad |\xi - (\sigma_2,0)| \leq \varepsilon^{\frac12}\},\\
& C_2 = \{ (\tau,\xi) \in \R^3 \, |\, |\tau-\sigma_2|\xi|^2|\leq \varepsilon, \quad |\xi - (\sigma_1,0)| \leq \varepsilon^{\frac12}\},\\
& C_3 = \{ (\tau,\xi) \in \R^3 \, |\, |\tau-\sigma_3|\xi|^2|\leq \varepsilon, \quad |\xi + (\sigma_1+\sigma_2,0)| \leq \varepsilon^{\frac12}\},
\end{align*}
then it holds that
\begin{equation}\label{est:remark3.4}
| (f_1 * f_2 * f_3)(0)| \sim \varepsilon^{\frac54 - \frac{1}{2p}} \|f_1\|_{L_{\tau,r}^2 L_{\theta}^{p}(S_1(\varepsilon))} \|f_2\|_{L^2(S_2(\varepsilon))} \|f_3\|_{L^2(S_3(\varepsilon))}.
\end{equation}
While, since $\sigma_1+\sigma_2 \not= 0$, \eqref{est:ThickenedLWrad} implies
\[
| (f_1 * f_2 * f_3)(0)| \lesssim \varepsilon^{\frac32 - \frac{1}{p}} \|f_1\|_{L_{\tau,r}^2 L_{\theta}^{p}(S_1(\varepsilon))} \|f_2\|_{L^2(S_2(\varepsilon))} \|f_3\|_{L^2(S_3(\varepsilon))}.
\]
If $2<p$, this clearly contradicts \eqref{est:remark3.4} by letting $\varepsilon \to 0$.

Notice that, in the study of the Schr\"{o}dinger system which is corresponding system to the phase functions $\varphi_1$, $\varphi_2$, $\varphi_3$, the relation $1/\sigma_1 + 1/\sigma_2 + 1 / \sigma_3 = 0$ is called mass resonance relation. 
The Schr\"{o}dinger system under mass resonance shows a significant difference from that under the non-mass resonance case $1/\sigma_1 + 1/\sigma_2 + 1 / \sigma_3 \not= 0$. 
For the works of Schr\"{o}dinger system under mass resonance, we refer to \cite{IKS13}, \cite{H14}, \cite{IKO16}, \cite{SaSu}, \cite{KS20}, and for the results under non-mass resonance case, see e.g., \cite{H14}, \cite{HirKin}, \cite{HKO}, \cite{HKO2}.
\end{enumerate}
\end{rem}

\begin{proof}[Proof of Theorem \ref{theorem:trilinearest}]
We first consider the case $N_3 \lesssim \varepsilon$. 
Since $|\supp f_3| \lesssim \varepsilon N_3^2$, for any $2 \leq p \leq \infty$, a simple computation gives
\begin{align*}
| (f_1 * f_2 * f_3)(0)| & \leq \|f_1 * f_2\|_{L^{\infty}} \|f_3\|_{L^1(S_3(\varepsilon))}\\
& \lesssim \varepsilon^{\frac12} N_3
\|f_1\|_{L^2(S_1(\varepsilon))} \|f_2\|_{L^2(S_2(\varepsilon))} \|f_3\|_{L^2(S_3(\varepsilon))}\\
& \leq \varepsilon^{\frac12} N_3
\|f_1\|_{L_{\tau,r}^2 L_{\theta}^{p}(S_1(\varepsilon))} \|f_2\|_{L^2(S_2(\varepsilon))} \|f_3\|_{L^2(S_3(\varepsilon))}.
\end{align*}
Since $N_3 \lesssim \varepsilon$ implies $\varepsilon^{\frac12} N_3 \lesssim C(\varepsilon, p , N_3)$ for any $2 \leq p \leq \infty$, this completes the proof of~\eqref{est:ThickenedLWrad} when $N_3 \lesssim \varepsilon$. 

We assume $N_3 \gg \varepsilon$. 
Without loss of generality, we may assume that $f_1$, $f_2$, $f_3$ are nonnegative. 
For $i=1,2,3$, because of the support condition of $f_i$, by applying harmless decomposition, there exist $c_i$, $d >0$ such that $c_i \sim N_i$, 
$d=d(\max_{i=1,2,3}\|\varphi_i\|_{C^{1,\beta}({B_2})})\ll1$ and we can replace the original $S_i(\varepsilon)$ defined in \eqref{defn:ThickenedS} by 
\begin{equation}\label{def:Thickenedhypersurfaces}
S_i(\varepsilon) = \{(\tau,\xi)\in \R^3 \, | \, |\tau - \varphi_i(\xi)| < \varepsilon, \ c_i \leq |\xi| \leq c_i+d \}.
\end{equation}
Let $A \in 2^{\N}$ satisfy $1 \lesssim A \lesssim (N_3/\varepsilon)^{\frac12}$. 
As in the proof of \cite[Theorem 4.3]{KS21}, we define a finitely overlapping family of balls with radius $\varepsilon$ covering $\R^3$ by $\{ B_{\varepsilon,j} \}_{j \in \mathbb{N}}$ and decompose the support of $f_3$ into $\{ B_{\varepsilon,j} \}_{j \in J_{3,\varepsilon}}$ where $J_{3,\varepsilon} = \{ j \in \N \, | \, B_{\varepsilon,j} \cap S_3(\varepsilon) \not= \emptyset \}$. We define
\begin{align*}
&S_{1,j}^A(\varepsilon) = \biggl\{(\tau_1,\xi_1) \in S_1(\varepsilon) \, \biggl| \, 
\begin{aligned}
& \mathrm{There \ exists} \ (\tau,\xi)\in B_{\varepsilon,j} \ \mathrm{such \ that} \\
& |\sin \angle(\xi_1,\xi)| \sim A^{-1} \ \mathrm{and} \  
 - (\tau,\xi) - (\tau_1,\xi_1) \in S_2(\varepsilon) 
\end{aligned} \ 
\biggr\},\\
& S_{2,j}^A(\varepsilon) = \biggl\{(\tau_2,\xi_2) \in S_2(\varepsilon) \, \biggl| \, 
\begin{aligned}
& \mathrm{There \ exists} \ (\tau,\xi)\in B_{\varepsilon,j} \ \mathrm{such \ that} \\
& |\sin \angle(\xi_2,\xi)| \sim A^{-1} \ \mathrm{and} \  - (\tau,\xi) - (\tau_2,\xi_2) \in S_1(\varepsilon) 
\end{aligned} \ 
\biggr\},\\
& \widetilde{S}_{1,j}^A(\varepsilon) = \biggl\{(\tau_1,\xi_1) \in S_1(\varepsilon) \, \biggl| \, 
\begin{aligned}
& \mathrm{There \ exists} \ (\tau,\xi)\in B_{\varepsilon,j} \ \mathrm{such \ that} \\
& |\sin \angle(\xi_1,\xi)| \lesssim A^{-1} \ \mathrm{and} \  - (\tau,\xi) - (\tau_1,\xi_1) \in S_2(\varepsilon) 
\end{aligned} \ 
\biggr\},\\
& \widetilde{S}_{2,j}^A(\varepsilon) = \biggl\{(\tau_2,\xi_2) \in S_2(\varepsilon) \, \biggl| \, 
\begin{aligned}
& \mathrm{There \ exists} \ (\tau,\xi)\in B_{\varepsilon,j} \ \mathrm{such \ that} \\
& |\sin \angle(\xi_2,\xi)| \lesssim A^{-1} \ \mathrm{and} \  - (\tau,\xi) - (\tau_2,\xi_2) \in S_1(\varepsilon) 
\end{aligned} \ 
\biggr\}.
\end{align*}
Let $A_0$ be a dyadic number such that $A_0 \sim (N_3/\varepsilon)^{\frac12}$. 
For $\xi=(\xi^{(1)}, \xi^{(2)}) \in \R^2$, we define $\xi^{\perp}  = (-\xi^{(2)}, \xi^{(1)})$. Note that if $|\xi_1| \sim |\xi+\xi_1|$, we have 
\[
|\sin \angle(\xi,\xi_1)| = \frac{|\xi^{\perp} \cdot \xi_1|}{|\xi| | \xi_1|}
 \sim \frac{|\xi^{\perp} \cdot (\xi+\xi_1)|}{|\xi| |\xi+ \xi_1|} = |\sin \angle (\xi, -(\xi+\xi_1))|.
\] 
Then, $(\tau,\xi) \in B_{\varepsilon, j}$ and $(\tau_1,\xi_1) \in S_{1,j}^A(\varepsilon)$ imply $(-\tau - \tau_1,-\xi- \xi_1) \in S_{2,j}^{A}(\varepsilon)$. 
Hence, we may decompose $| (f_1 * f_2 * f_3)(0)|$ as
\begin{align*}
 | (f_1 * f_2 * f_3)(0)|  \lesssim & \sum_{1 \lesssim A \lesssim A_0} \sum_{j \in J_{3,\varepsilon}} 
 \bigl| \bigl( f_1|_{S_{1,j}^A(\varepsilon)}* f_2|_{S_{2,j}^A(\varepsilon)} * f_3|_{B_{\varepsilon,j}}\bigr)(0)\bigr|\\
& + \sum_{j \in J_{3,\varepsilon}} 
 \bigl| \bigl( f_1|_{\widetilde{S}_{1,j}^{A_0}(\varepsilon)}* f_2|_{\widetilde{S}_{2,j}^{A_0}(\varepsilon)} * f_3|_{B_{\varepsilon,j}} \bigr)(0)\bigr|.
\end{align*}
Since $|B_{\varepsilon,j}| \sim \varepsilon^3$, the Cauchy-Schwarz inequality yields
\begin{align}
\label{est:conv-0.1}
\begin{split}
\bigl| \bigl( f_1|_{S_{1,j}^A(\varepsilon)}*  f_2|_{S_{2,j}^A(\varepsilon)} * & f_3|_{B_{\varepsilon,j}}\bigr)(0)\bigr|\\ & \lesssim \varepsilon^{\frac32} \bigl\| f_1|_{S_{1,j}^A(\varepsilon)}\bigr\|_{L^2} \bigl\| f_2|_{S_{2,j}^A(\varepsilon)}\bigr\|_{L^2} \bigl\|f_3|_{B_{\varepsilon,j}}\bigr\|_{L^2},
\end{split}\\
\label{est:conv-0.2}
\begin{split}
\bigl| \bigl( f_1|_{\widetilde{S}_{1,j}^{A_0}(\varepsilon)}*  f_2|_{\widetilde{S}_{2,j}^{A_0}(\varepsilon)} * & f_3|_{B_{\varepsilon,j}} \bigr)(0)\bigr|\\
& \lesssim \varepsilon^{\frac32} \bigl\| f_1|_{\widetilde{S}_{1,j}^{A_0}(\varepsilon)}\bigr\|_{L^2} \bigl\| f_2|_{\widetilde{S}_{2,j}^{A_0}(\varepsilon)}\bigr\|_{L^2} \bigl\|f_3|_{B_{\varepsilon,j}}\bigr\|_{L^2}.
\end{split}
\end{align}
Because of \eqref{est:conv-0.1} and \eqref{est:conv-0.2}, it suffices to show the following:
\begin{align}\label{est:sum-A-0.1}
\begin{split}
 \sum_{j \in J_{3,\varepsilon}}  \bigl\| f_1|_{S_{1,j}^A(\varepsilon)} & \bigr\|_{L^2} \bigl\| f_2|_{S_{2,j}^A(\varepsilon)}\bigr\|_{L^2}  \bigl\|f_3|_{B_{\varepsilon,j}}\bigr\|_{L^2}\\
& \lesssim \Bigl( \frac{N_3}{\varepsilon A} \Bigr)^{\frac{1}{p}} \|f_1\|_{L_{\tau,r}^2 L_{\theta}^{p}(S_1(\varepsilon))} \|f_2\|_{L^2(S_2(\varepsilon))} \|f_3\|_{L^2(S_3(\varepsilon))},
\end{split}\\
\label{est:sum-A-0.2}
\begin{split}
\sum_{j \in J_{3,\varepsilon}} \bigl\| f_1|_{\widetilde{S}_{1,j}^{A_0}(\varepsilon)} & \bigr\|_{L^2} \bigl\| f_2|_{\widetilde{S}_{2,j}^{A_0}(\varepsilon)}\bigr\|_{L^2}  \bigl\|f_3|_{B_{\varepsilon,j}}\bigr\|_{L^2}\\
&\lesssim  \Bigl( \frac{N_3}{\varepsilon} \Bigr)^{\frac{1}{2p}} \|f_1\|_{L_{\tau,r}^2 L_{\theta}^{p}(S_1(\varepsilon))} \|f_2\|_{L^2(S_2(\varepsilon))} \|f_3\|_{L^2(S_3(\varepsilon))}.
\end{split}
\end{align}

To show these estimates, we observe widths of the sets $S_{1,j}^A(\varepsilon)$, $S_{2,j}^A(\varepsilon)$. Let us fix $A$, $j \in J_{3,\varepsilon}$ of $S_{1,j}^A(\varepsilon)$ and write $(\tau_1,\xi_1) = (\tau_1,r_1 \cos \theta_1, r_1 \sin \theta_1)$. 
\begin{claim}\label{claim1}
If $r_1$ is fixed and $(\tau_1, r_1 \cos \theta_1, r_1 \sin \theta_1) \in S_{1,j}^A(\varepsilon)$, then $\theta_1$ is confined to at most four intervals of length $\sim A \varepsilon/N_3$.
\end{claim}
\begin{proof}[Proof of Claim~\ref{claim1}]
We recall that $(\tau_1,\xi_1) \in S_{1,j}^A(\varepsilon)$ means $(\tau_1,\xi_1) \in S_1(\varepsilon)$ and the existence of $(\tau, \xi) \in B_{\varepsilon,j}$ such that $(-\tau-\tau_1, -\xi-\xi_1) \in S_2(\varepsilon)$, which yield
\[
|\varphi_1(\xi_1) + \varphi_2(\xi+\xi_1)+\varphi_3(\xi)| \lesssim \varepsilon.
\]
Thus, since $\varphi_i \in C^{1,\beta}({B_2})$ and the radius of $B_{\varepsilon,j}$ is $\varepsilon$, it suffices to see that for fixed $r_1>0$ it holds that
\[
\bigl|\partial_{\theta_1} \bigl( \varphi_1 (\xi_1) +  \varphi_2 (\xi +\xi_1) \bigr) \bigr| \sim A^{-1} N_3.
\]
For a radial function $\varphi:\R^2 \to \R$, we define $\widetilde{\varphi} (r) = \varphi(r,0)$ for $r>0$. 
If $\varphi:\R^2 \to \R$ is $C^{1}$ and radial, for $\xi \in \R^2$, it holds that
\[
\nabla \varphi(\xi) = 
 \frac{\widetilde{\varphi}'(|\xi|)}{|\xi|}\xi=
  \mathrm{sgn} \, \widetilde{\varphi}'(|\xi|) \frac{|\nabla \varphi(\xi)|}{|\xi|}\xi.
\] 
Thus, it follows from \ref{assumption1-Ass2} and \ref{assumption2-Ass2} in \textit{Assumption}~\ref{assumption:hypersurface} that
\begin{align*}
\bigl|\partial_{\theta_1} \bigl( \varphi_1 (\xi_1) +  \varphi_2 (\xi +\xi_1) \bigr) \bigr|& = 
|\xi_1^{\perp} \cdot \nabla \varphi_2(\xi + \xi_1)|\\
& = \frac{| \nabla\varphi_2(\xi+\xi_1)|}{|\xi+\xi_1|} |\xi_1^{\perp} \cdot (\xi+\xi_1) |\\
& = \frac{|\xi| \, |\xi_1|}{|\xi+\xi_1|}| \nabla\varphi_2(\xi+\xi_1)| |\sin \angle(\xi, \xi_1)|\sim A^{-1} N_3.
\end{align*}
\end{proof}
Similarly, for fixed $A$, $j \in J_{3,\varepsilon}$, $r_1>0$, if $(\tau_2, r_2 \cos \theta_2, r_2 \sin \theta_2) \in S_{2,j}^A(\varepsilon)$ then $\theta_2$ is confined to at most four intervals of length comparable to $A \varepsilon/N_3$. 

For $\ell \in \N$, 
we define $D_{\ell}= \{(\tau,\xi)\, | \, \ell \varepsilon \leq |\xi| \leq (\ell+1) \varepsilon \}$ and 
$J_{3, \varepsilon,\ell} = \{j \in J_{3,\varepsilon} \,| \, B_{\varepsilon,j} \cap D_{\ell} \not= \emptyset \}$. 
We put
$\mathcal{L}_{S_3} = \{ \ell \in \N \,| \, S_3(\varepsilon)\cap D_{\ell} \not= \emptyset\}$. 
For $\ell \in \mathcal{L}_{S_3}$, we define
\begin{align*}
T_{1,\ell}^A(\varepsilon) = \biggl\{ (\tau_1,\xi_1) \in S_1(\varepsilon) \, \biggl| \ \, 
\begin{aligned}
& \mathrm{There \ exists} \ (\tau_1', \xi_1') \in {\bigcup}_{j \in J_{3,\varepsilon,\ell}}S_{1,j}^{A}(\varepsilon) \\
& \mathrm{such \ that} \ |\xi_1|= |\xi_1'| 
\end{aligned} \, \biggr\},\\
T_{2,\ell}^A(\varepsilon) = \biggl\{ (\tau_2,\xi_2) \in S_2(\varepsilon) \, \biggl| \ \, 
\begin{aligned}
& \mathrm{There \ exists} \  (\tau_2', \xi_2') \in {\bigcup}_{j \in J_{3,\varepsilon,\ell}}S_{2,j}^{A}(\varepsilon) \\
& \mathrm{such \ that} \ |\xi_2|= |\xi_2'|
\end{aligned} \, \biggr\}.
\end{align*}
Note that $(\tau,\xi) \in T_{i,\ell}^A(\varepsilon)$ and $|\xi| = |\eta|$ imply $(\tau,\eta) \in T_{i,\ell}^A(\varepsilon)$ $(i=1,2)$. 
For fixed $j \in J_{3,\varepsilon}$, it follows from Claim~\ref{claim1} that
\begin{equation}\label{est:WidthofS1jA}
\begin{split}
\| f_1|_{S_{1,j}^A(\varepsilon)} \|_{L^2} & \leq \| \textbf{1}_{S_{1,j}^A(\varepsilon)} \|_{L_{\tau,r}^{\infty} L_{\theta}^{\frac{2p}{p-2}}} \|f_1 |_{S_{1,j}^A(\varepsilon)}\|_{L_{\tau,r}^2 L_{\theta}^{p}}\\
&  \lesssim \Bigl( \frac{A \varepsilon}{N_3} \Bigr)^{\frac12-\frac{1}{p}}\|f_1 |_{S_{1,j}^A(\varepsilon)}\|_{L_{\tau,r}^2 L_{\theta}^{p}}.
\end{split}
\end{equation}
Similarly, because of the width restriction of $S_{2,j}^A(\varepsilon)$ as in Claim~\ref{claim1}, for fixed $\ell \in \mathcal{L}_{S_3}$, it holds that
\[
\sum_{j \in J_{3,\varepsilon,\ell}} \mathbf{1}_{S_{2,j}^A(\varepsilon)} \lesssim A.
\]
This and $\bigcup_{j \in J_{3,\varepsilon,\ell}} S_{2,j}^A(\varepsilon) \subset T_{2,\ell}^A(\varepsilon)$ yield
\begin{equation}\label{est:Sumofj}
\Bigl( \sum_{j \in J_{3,\varepsilon,\ell}}  \bigl\| f_2|_{S_{2,j}^A(\varepsilon)}\bigr\|_{L^2}^2 \Bigr)^{\frac12} \lesssim A^{\frac12}\bigl\| f_2|_{T_{2,\ell}^A(\varepsilon)}\bigr\|_{L^2}.
\end{equation}
Consequently, by using the Cauchy-Schwarz inequality and \eqref{est:WidthofS1jA}, \eqref{est:Sumofj}, we have
\begin{align*}
& \sum_{j \in J_{3,\varepsilon}} \bigl\| f_1|_{S_{1,j}^A(\varepsilon)}\bigr\|_{L^2} \bigl\| f_2|_{S_{2,j}^A(\varepsilon)}\bigr\|_{L^2} \bigl\|f_3|_{B_{\varepsilon,j}}\bigr\|_{L^2}\\
\lesssim & \sum_{\ell \in \mathcal{L}_{S_3}} \sum_{j \in J_{3,\varepsilon,\ell}} \bigl\| f_1|_{S_{1,j}^A(\varepsilon)}\bigr\|_{L^2} \bigl\| f_2|_{S_{2,j}^A(\varepsilon)}\bigr\|_{L^2} \bigl\|f_3|_{B_{\varepsilon,j}}\bigr\|_{L^2}\\
\lesssim & \sum_{\ell \in \mathcal{L}_{S_3}}  \sup_{j \in J_{3,\varepsilon,\ell}} \bigl\| f_1|_{S_{1,j}^A(\varepsilon)} \bigl\|_{L^2} \Bigl( \sum_{j \in J_{3,\varepsilon,\ell}}  \bigl\| f_2|_{S_{2,j}^A(\varepsilon)}\bigr\|_{L^2}^2 \Bigr)^{\frac12} \Bigl( \sum_{j \in J_{3,\varepsilon,\ell}}  \bigl\|f_3|_{B_{\varepsilon,j}}\bigr\|_{L^2}^2 \Bigr)^{\frac12}\\
\lesssim & \sum_{\ell \in \mathcal{L}_{S_3}}A^{1-\frac{1}{p}} \Bigl( \frac{\varepsilon}{N_3} \Bigr)^{\frac12-\frac{1}{p}}\bigl\| f_1|_{T_{1,\ell}^A(\varepsilon)} \bigl\|_{L_{\tau,r}^2 L_{\theta}^{p}} 
\bigl\| f_2|_{T_{2,\ell}^A(\varepsilon)}\bigr\|_{L^2} \Bigl( \sum_{j \in J_{3,\varepsilon,\ell}}  \bigl\|f_3|_{B_{\varepsilon,j}}\bigr\|_{L^2}^2 \Bigr)^{\frac12}\\
\lesssim & A^{1-\frac{1}{p}} \Bigl( \frac{\varepsilon}{N_3} \Bigr)^{\frac12-\frac{1}{p}} \Bigl( \sum_{\ell \in \mathcal{L}_{S_3}}\bigl\| f_1|_{T_{1,\ell}^A(\varepsilon)} \bigl\|_{L_{\tau,r}^2 L_{\theta}^{p}}^2 
\bigl\| f_2|_{T_{2,\ell}^A(\varepsilon)}\bigr\|_{L^2}^2 \Bigr)^{\frac12} \|f_3 \|_{L^2(S_{3(\varepsilon)})}.
\end{align*}
Hence, the estimate \eqref{est:sum-A-0.1} follows from 
\begin{equation}\label{est:orthogonality-0.7}
\sup_{\substack{(\tau_1,\xi_1) \in S_1(\varepsilon) \\(\tau_2,\xi_2) \in S_2(\varepsilon)}} \sum_{\ell \in \mathcal{L}_{S_3}} \mathbf{1}_{T_{1,\ell}^A(\varepsilon)}(\tau_1,\xi_1) \mathbf{1}_{T_{2,\ell}^A(\varepsilon)} (\tau_2,\xi_2) \lesssim \frac{N_3}{ A^2 \varepsilon}.
\end{equation}
Because the number of $\ell$ in $\mathcal{L}_{S_3}$ is comparable to $N_3 / \varepsilon$, we assume $A \gg 1$.

We prove \eqref{est:orthogonality-0.7} by contradiction. 
Suppose that, for some fixed $(\tau_1,\xi_1) \in S_1(\varepsilon)$ and $(\tau_2,\xi_2) \in S_2(\varepsilon)$, there exist $\ell_1$, $\ell_2 \in \mathcal{L}_{S_3}$ such that $A^{-2}N_3 \ll (\ell_1-\ell_2)\varepsilon \leq \min(N_3,d)$ 
(notice that this gives $\ell_1-\ell_2 \gg 1$ since $A \lesssim (N_3/\varepsilon)^{\frac12}$) where $d \ll 1$ is defined in \eqref{def:Thickenedhypersurfaces} and 
\[
(\tau_1,\xi_1) \in T_{1,\ell_1}^A(\varepsilon) \cap T_{1,\ell_2}^A(\varepsilon), \quad (\tau_2,\xi_2) \in T_{2,\ell_1}^A(\varepsilon) \cap T_{2,\ell_2}^A(\varepsilon).
\]
We deduce from $(\tau_1,\xi_1) \in T_{1,\ell_1}^A(\varepsilon)$ that there exists $(\tau', \xi') \in S_{3}(4\varepsilon)$ such that $||\xi'|-\ell_1 \varepsilon| \leq 2 \varepsilon$, $|\sin \angle(\xi_1,\xi')| \sim A^{-1}$ and $-(\tau',\xi')-(\tau_1,\xi_1) \in S_2(4\varepsilon)$. 
Similarly, because $(\tau_1,\xi_1) \in T_{1,\ell_2}^A(\varepsilon)$, we may find $(\widetilde{\tau}, \widetilde{\xi}) \in S_{3}(4\varepsilon)$ which satisfy $||\widetilde{\xi}|-\ell_2 \varepsilon| \leq 2 \varepsilon$, $|\sin \angle(\xi_1,\widetilde{\xi})| \sim A^{-1}$ and $-(\widetilde{\tau},\widetilde{\xi})-(\tau_1,\xi_1) \in S_2(4\varepsilon)$. 
We observe that $(\tau', \xi') \in S_{3}(4\varepsilon)$ and $(\widetilde{\tau}, \widetilde{\xi}) \in S_{3}(4\varepsilon)$ give $|(\tau'-\widetilde{\tau})-(\varphi_3(\xi') - \varphi_3(\widetilde{\xi}))| \lesssim \varepsilon$. 
In addition, $-(\tau',\xi')-(\tau_1,\xi_1) \in S_2(4\varepsilon)$ and $-(\widetilde{\tau},\widetilde{\xi})-(\tau_1,\xi_1) \in S_2(4\varepsilon)$ yield $|(\tau'- \widetilde{\tau})+ (\varphi_2(\xi'+\xi_1)-\varphi_2(\widetilde{\xi}+\xi_1))| \lesssim \varepsilon$. 
Therefore, it holds that
\begin{equation}\label{est:normal-0.1}
|(\varphi_2(\xi'+\xi_1)-\varphi_2(\widetilde{\xi}+\xi_1))+(\varphi_3(\xi') - \varphi_3(\widetilde{\xi}))| \lesssim \varepsilon.
\end{equation}
\begin{claim}\label{claim2}
The estimate \eqref{est:normal-0.1} implies
\begin{equation}\label{est:Nablaphi2phi3}
|\nabla \varphi_2 (-\xi'-\xi_1)- \nabla \varphi_3(\xi')| \ll 1.
\end{equation}
\end{claim}
\begin{proof}[Proof of Claim~\ref{claim2}]
Since $\varphi_2$ is radial, \eqref{est:Nablaphi2phi3} is equivalent to
\begin{equation}\label{est:Nablaphi2phi3-2}
|\nabla \varphi_2 (\xi'+\xi_1)+ \nabla \varphi_3(\xi')| \ll 1.
\end{equation}
Recall that $|\sin \angle(\xi_1,\xi')| \sim |\sin \angle(\xi_1,\widetilde{\xi})|\sim A^{-1}$. 
To prove \eqref{est:Nablaphi2phi3-2}, we first see that \eqref{est:normal-0.1} provides either $|\angle(\xi_1,\xi')| \sim |\angle(\xi_1,\widetilde{\xi})| \sim A^{-1}$ or $|\angle(\xi_1,\xi')- \pi| \sim |\angle(\xi_1,\widetilde{\xi})-\pi| \sim A^{-1}$. 
We will observe that $|\angle(\xi_1,\xi')| \sim |\angle(\xi_1,\widetilde{\xi})-\pi| \sim A^{-1}$ contradicts \eqref{est:normal-0.1}. In the same way, we may see that $|\angle(\xi_1,\xi')- \pi| \sim |\angle(\xi_1,\widetilde{\xi})| \sim A^{-1}$ contradicts \eqref{est:normal-0.1}. 
Since $A\gg 1$, it is easy to see that $|\angle(\xi_1,\xi')| \sim |\angle(\xi_1,\widetilde{\xi})-\pi| \sim A^{-1}$ implies
\begin{equation}\label{est:difference-01}
|\xi' +\xi_1| - |\widetilde{\xi} + \xi_1| \gtrsim N_3.
\end{equation}
Let us use the function $\widetilde{\varphi}_{i} (r) = \varphi_i(r,0)$. It follows from \ref{assumption1-Ass2} and \ref{assumption2-Ass2} in \textit{Assumption}~\ref{assumption:hypersurface} that $|\widetilde{\varphi}_i'(r)| \sim 1$ if $r \sim 1$. 
Hence, we deduce from \eqref{est:difference-01} that
\begin{equation}\label{est:difference-02}
|\varphi_2(\xi'+\xi_1)-\varphi_2(\widetilde{\xi}+\xi_1)| = |\widetilde{\varphi}_2(|\xi'+\xi_1|)-\widetilde{\varphi}_2(|\widetilde{\xi}+\xi_1|)| \gtrsim N_3.
\end{equation}
Here we used $|\xi' +\xi_1| \sim |\widetilde{\xi} + \xi_1| \sim 1$ which follows from $-(\tau',\xi')-(\tau_1,\xi_1) \in S_2(4\varepsilon)$, $-(\widetilde{\tau},\widetilde{\xi})-(\tau_1,\xi_1) \in S_2(4\varepsilon)$. 
While, since $||\xi'|-\ell_1 \varepsilon| \leq 2 \varepsilon$, $||\widetilde{\xi}|-\ell_2 \varepsilon| \leq 2 \varepsilon$, and $(\ell_1-\ell_2)\varepsilon \ll N_3$, we have
\begin{equation}\label{est:difference-03}
|\varphi_3(\xi') - \varphi_3(\widetilde{\xi})| \lesssim (\ell_1-\ell_2)\varepsilon \ll N_3.
\end{equation}
Clearly, since $\varepsilon \ll N_3$, the two estimates \eqref{est:difference-02} and \eqref{est:difference-03} contradict \eqref{est:normal-0.1}. 
Consequently, either $|\angle(\xi_1,\xi')| \sim |\angle(\xi_1,\widetilde{\xi})| \sim A^{-1}$ or $|\angle(\xi_1,\xi')- \pi| \sim |\angle(\xi_1,\widetilde{\xi})-\pi| \sim A^{-1}$ holds. 

Now we consider \eqref{est:Nablaphi2phi3-2}. 
Let us consider the case $|\angle(\xi_1,\xi')| \sim |\angle(\xi_1,\widetilde{\xi})| \sim A^{-1}$ first. 
A simple calculation yields
\[
0 \leq |\xi'| + |\xi_1| - |\xi'+\xi_1| \lesssim A^{-2} N_3, \quad 0 \leq |\widetilde{\xi}|+|\xi_1| - |\widetilde{\xi} + \xi_1| \lesssim A^{-2} N_3,
\]
which imply
\begin{equation}\label{est:difference-04}
\bigl| (|\xi'+\xi_1|-|\widetilde{\xi} + \xi_1|) - (|\xi'|-|\widetilde{\xi}|) \bigr| \lesssim A^{-2} N_3.
\end{equation}
Note that $|\xi'|-|\widetilde{\xi}| \sim (\ell_1-\ell_2)\varepsilon >0$. 
By using \eqref{est:normal-0.1} and \eqref{est:difference-04}, since $\varphi_2$, $\varphi_3$ are radial, $\varphi_2 \in C^{1,\beta}({B_2})$, and $A^{-2}N_3 \ll (\ell_1-\ell_2)\varepsilon$, we get
\begin{align}
\notag
&\biggl| \frac{{\varphi}_2(\xi'+\xi_1)- {\varphi}_2(\widetilde{\xi}+\xi_1)}{ |\xi'+\xi_1|-|\widetilde{\xi}+\xi_1|} +\frac{{\varphi}_3(\xi') - {\varphi}_3(\widetilde{\xi})}{|\xi'|-|\widetilde{\xi}|} \biggr|\\
\notag
 \leq & 
\biggl| \frac{{\varphi}_2(\xi'+\xi_1)- {\varphi}_2(\widetilde{\xi}+\xi_1)}{ |\xi'+\xi_1|-|\widetilde{\xi}+\xi_1|} - 
\frac{{\varphi}_2(\xi'+\xi_1)- {\varphi}_2(\widetilde{\xi}+\xi_1)}{|\xi'|-|\widetilde{\xi}|}\biggr| \\
& \qquad \notag 
+ \frac{\bigl|(\varphi_2(\xi'+\xi_1)-\varphi_2(\widetilde{\xi}+\xi_1))+(\varphi_3(\xi') - \varphi_3(\widetilde{\xi}))\bigr|}{|\xi'|-|\widetilde{\xi}|}\\
\notag 
\lesssim & \biggl|1- \frac{ |\xi'+\xi_1|-|\widetilde{\xi}+\xi_1|}{|\xi'|-|\widetilde{\xi}|} \biggr|+\frac{
|(\varphi_2(\xi'+\xi_1)-\varphi_2(\widetilde{\xi}+\xi_1))+(\varphi_3(\xi') - \varphi_3(\widetilde{\xi}))|}{(\ell_1-\ell_2)\varepsilon}\\
\label{est:difference-04.5}
\lesssim &  \frac{A^{-2}N_3 + \varepsilon}{(\ell_1-\ell_2)\varepsilon} \ll 1.
\end{align}

Here, since $\varphi_2$ and $\varphi_3$ are radial, we may assume that $\xi_1$ is on the right half of the first axis, i.e. $\xi_1 = (|\xi_1|,0)$. 
In addition to \eqref{est:difference-04.5}, we deduce from $\xi_1 = (|\xi_1|,0)$, $|\angle(\xi_1,\xi')| \sim |\angle(\xi_1,\widetilde{\xi})| \sim A^{-1}$, and $\varphi_2$, $\varphi_3 \in C^{1,\beta}({B_2})$ that
\begin{align}\label{est:difference-05}
& \biggl| \nabla \varphi_2(\xi'+\xi_1) - \Bigl( \frac{{\varphi}_2(\xi'+\xi_1)- {\varphi}_2(\widetilde{\xi}+\xi_1)}{ |\xi'+\xi_1|-|\widetilde{\xi}+\xi_1|},0\Bigr) \biggr|\ll 1,\\
\label{est:difference-06}
& \biggl| \nabla \varphi_3(\xi') - \Bigl( \frac{{\varphi}_3(\xi') - {\varphi}_3(\widetilde{\xi})}{|\xi'|-|\widetilde{\xi}|} , 0\Bigr) \biggr| \ll 1.
\end{align}
Here we used $|\xi'|-|\widetilde{\xi}| \sim (\ell_1-\ell_2)\varepsilon \ll 1$, and $|\xi'+\xi_1|-|\widetilde{\xi}+\xi_1| \ll1$ which follows from $|\xi'|-|\widetilde{\xi}| \ll1$ and \eqref{est:difference-04}. 
Consequently, by the triangle inequality, the estimates \eqref{est:difference-04.5}, \eqref{est:difference-05}, \eqref{est:difference-06} give \eqref{est:Nablaphi2phi3-2}.

Next, we consider the case $|\angle(\xi_1,\xi')- \pi| \sim |\angle(\xi_1,\widetilde{\xi})-\pi| \sim A^{-1}$. 
Similarly to the above case, we assume that $\xi_1 = (|\xi_1|,0)$. Let $\xi^{(1)} \in \R$ denote the first component of $\xi \in \R^2$. We divide the proof into the two cases. 
The first case is ${\xi'}^{(1)} + \xi_1^{(1)} >0$, $\widetilde{\xi}^{(1)} + \xi_1^{(1)} >0$, and the second is 
${\xi'}^{(1)} + \xi_1^{(1)} <0$, $\widetilde{\xi}^{(1)} + \xi_1^{(1)} <0$. 
In the case ${\xi'}^{(1)} + \xi_1^{(1)} >0$, $\widetilde{\xi}^{(1)} + \xi_1^{(1)} >0$, since $|\angle(\xi_1,\xi')-\pi| \sim |\angle(\xi_1,\widetilde{\xi})-\pi| \sim A^{-1}$, we obtain
\[
0\leq |\xi'+\xi_1| - (|\xi_1| -|\xi'|)  \lesssim A^{-2} N_3, \quad  
0 \leq |\widetilde{\xi} + \xi_1|  -( |\xi_1| -|\widetilde{\xi}|) \lesssim A^{-2} N_3,
\]
which give
\[
\bigl| (|\xi'+\xi_1|-|\widetilde{\xi} + \xi_1|) + (|\xi'|-|\widetilde{\xi}|) \bigr| \lesssim A^{-2} N_3.
\]
In a similar way as in the proof of \eqref{est:difference-04.5}, this estimate provides
\begin{equation}
\label{est:difference-07}
\biggl| \frac{{\varphi}_2(\xi'+\xi_1)- {\varphi}_2(\widetilde{\xi}+\xi_1)}{ |\xi'+\xi_1|-|\widetilde{\xi}+\xi_1|} -\frac{{\varphi}_3(\xi') - {\varphi}_3(\widetilde{\xi})}{|\xi'|-|\widetilde{\xi}|} \biggr| \ll 1.
\end{equation}
Note that, the assumptions $|\angle(\xi_1,\xi')- \pi| \sim |\angle(\xi_1,\widetilde{\xi})-\pi| \sim A^{-1}$ with $\xi_1 = (|\xi_1|,0)$ mean that ${\xi'}^{(1)}<0$, $\widetilde{\xi}^{(1)} < 0$. 
Hence, it follows from $\xi_1 = (|\xi_1|,0)$, $|\angle(\xi_1,\xi')- \pi| \sim |\angle(\xi_1,\widetilde{\xi})-\pi| \sim A^{-1}$, $|\xi'|-|\widetilde{\xi}| \ll 1$ and $|\widetilde{\xi}+\xi_1| - |\xi'+\xi_1| \ll 1$ that
\begin{align}\label{est:difference-09}
& \biggl| \nabla \varphi_2(\xi'+\xi_1) - \Bigl( \frac{{\varphi}_2(\xi'+\xi_1)- {\varphi}_2(\widetilde{\xi}+\xi_1)}{ |\xi'+\xi_1|-|\widetilde{\xi}+\xi_1|},0\Bigr) \biggr|\ll 1,\\
\label{est:difference-09}
& \biggl| \nabla \varphi_3(\xi') + \Bigl( \frac{{\varphi}_3(\xi') - {\varphi}_3(\widetilde{\xi})}{|\xi'|-|\widetilde{\xi}|} , 0\Bigr) \biggr| \ll 1.
\end{align}
Consequently, by combining \eqref{est:difference-07}-\eqref{est:difference-09}, we complete the proof of \eqref{est:Nablaphi2phi3-2} for the case ${\xi'}^{(1)} + \xi_1^{(1)} >0$, $\widetilde{\xi}^{(1)} + \xi_1^{(1)} >0$. The remaining case ${\xi'}^{(1)} + \xi_1^{(1)} <0$, $\widetilde{\xi}^{(1)} + \xi_1^{(1)} <0$ can be handled in a similar way. We omit the details.
\end{proof}
Now, by using \eqref{est:Nablaphi2phi3}, we conclude \eqref{est:orthogonality-0.7} by contradiction. As mentioned in \eqref{def:Thickenedhypersurfaces}, the spatial variable of $S_i(\varepsilon)$ is confined to annulus of small width. Thus, by the definition of $T_{2,\ell}^A(\varepsilon)$, it follows from $(\tau_2,\xi_2) \in T_{2,\ell_1}^A(\varepsilon) \cap T_{2,\ell_2}^A(\varepsilon)$ that we may find $\xi_2' \in \R^2$ such that $|\xi_2'|=|\xi_2|$, $(\tau_2,\xi_2') \in T_{2,\ell_1}^A(\varepsilon) \cap T_{2,\ell_2}^A(\varepsilon)$ and $|\xi_1+\xi_2'+\xi'| \ll 1$. 
In the same way as above, it follows from $(\tau_2,\xi_2') \in T_{2,\ell_1}^A(\varepsilon) \cap T_{2,\ell_2}^A(\varepsilon)$ that there exists $(\tau_0,\eta) \in S_3(4\varepsilon)$ which satisfy $||\eta|-\ell_1 \varepsilon|\leq 2 \varepsilon$, $-(\tau_0,\eta)-(\tau_2,\xi_2') \in S_1(4\varepsilon)$, $|\sin \angle(\eta, \xi_2')| \sim A^{-1}$ and
\begin{equation}\label{est:Nablaphi1phi3-01}
|\nabla \varphi_1 (-\eta-\xi_2')- \nabla \varphi_3(\eta)| \ll 1.
\end{equation}
Since $||\xi'|-|\eta||\leq 4 \varepsilon$, $|\sin \angle(\xi', - \xi_1-\xi')| \sim |\sin \angle(\eta, \xi_2')| \lesssim A^{-1}$, and $|\xi_1+\xi_2'+\xi'| \ll 1$, it is easily confirmed that $|\xi' - \eta|+|\xi_1+\xi_2' + \eta|\ll 1$. Hence, because $\varphi_1$, $\varphi_3 \in C^{1,\beta}({B_2})$, \eqref{est:Nablaphi1phi3-01} implies
\begin{equation}\label{est:Nablaphi1phi3-02}
|\nabla \varphi_1 (\xi_1)- \nabla \varphi_3(\xi')| \ll 1.
\end{equation}
Consequently, \eqref{est:Nablaphi2phi3} and \eqref{est:Nablaphi1phi3-02} give
\[
|\nabla\varphi_1(\xi_1)-\nabla\varphi_3(\xi')| + |\nabla \varphi_2(-\xi'-\xi_1) - \nabla \varphi_3(\xi')|\ll1.
\]
Recall that $(\tau_1 ,\xi_1) \in T_{1,\ell_1}^A(\varepsilon) \subset S_1(\varepsilon)$, $(\tau', \xi') \in S_{3}(4\varepsilon)$, $-(\tau',\xi')-(\tau_1,\xi_1) \in S_2(4\varepsilon)$. Thus, this contradicts the assumption \ref{assumption3-Ass2} in \textit{Assumption}~\ref{assumption:hypersurface}.

Lastly, we consider \eqref{est:sum-A-0.2}. Following the proof of \eqref{est:sum-A-0.1}, it suffices to prove
\[
\sup_{\substack{(\tau_1,\xi_1) \in S_1(\varepsilon) \\(\tau_2,\xi_2) \in S_2(\varepsilon)}} \sum_{\ell \in \mathcal{L}_{S_3}} \mathbf{1}_{\widetilde{T}_{1,\ell}^{A_0}(\varepsilon)}(\tau_1,\xi_1) \mathbf{1}_{\widetilde{T}_{2,\ell}^{A_0}(\varepsilon)} (\tau_2,\xi_2) \lesssim 1,
\]
where
\begin{align*}
\widetilde{T}_{1,\ell}^{A_0}(\varepsilon) = \biggl\{ (\tau_1,\xi_1) \in S_1(\varepsilon) \, \biggl| \ \ 
\begin{aligned}
& \mathrm{There \ exists} \ (\tau_1', \xi_1') \in {\bigcup}_{j \in J_{3,\varepsilon,\ell}}\widetilde{S}_{1,j}^{A_0}(\varepsilon) \\
& \mathrm{such \ that} \ |\xi_1|= |\xi_1'| 
\end{aligned} \ \biggr\},\\
\widetilde{T}_{2,\ell}^{A_0}(\varepsilon) = \biggl\{ (\tau_2,\xi_2) \in S_2(\varepsilon) \, \biggl| \ \ 
\begin{aligned}
& \mathrm{There \ exists} \  (\tau_2', \xi_2') \in {\bigcup}_{j \in J_{3,\varepsilon,\ell}}\widetilde{S}_{2,j}^{A_0}(\varepsilon) \\
& \mathrm{such \ that} \ |\xi_2|= |\xi_2'|
\end{aligned} \ \biggr\}.
\end{align*}
This can be shown in the same way as for \eqref{est:orthogonality-0.7}. We omit the details.
\end{proof}
\section{Proof of the bilinear estimate}
In this section, we prove Proposition~\ref{prop2.2}.

It is known that the standard Strichartz estimate (see \cite{KT98}) provides the following estimate. 
\begin{lem}[Lemma\ 2.3 in \cite{GTV97}]\label{Bo_Stri}
Let $L\in 2^{\N_0}$, $u \in L^{2}(\R\times \R^{2})$, and $(q,r)$ be an admissible pair for the Schr\"odinger equation with $d=2$, i.e. $q > 2$, 
$\frac{2}{q} =2(\frac{1}{2}-\frac{1}{r})$. 
Then, we have
\[
\|Q_{L}u\|_{L_{t}^{q}L_{x}^{r}}\lesssim L^{\frac{1}{2}}\|Q_{L}u\|_{L^{2}_{t,x}},
\]
where $Q_L$ is defined in Definition~\ref{definition2.1}.
\end{lem}
By using Theorems~\ref{theorem:trilinearest}, we establish the trilinear estimate which plays a central role in the proof of Proposition~\ref{prop2.2}.
\begin{prop}\label{prop4.4}
For $i=1,2,3$, let $L_i$, $M_i$, $N_i \in 2^{\N_0}$, and $u_i \in L^2(\R \times \R^2)$. We define 
\[
v_{i}  = Q_{L_i}H_{M_i} P_{N_i} u_i,
\]
where $H_{M}$ is defined in \eqref{definition:H_M}. Then, for any $2 \leq p < \infty$, we have
\begin{equation}
\label{est:goal-prop4.3}
\Bigl| \int v_1 \overline{v}_2 \overline{v}_3 dt dx \Bigr| \lesssim (L_1L_2L_3)^{\frac12} M_{\min}^{\frac12 - \frac1p}N_{\max}^{-1+\frac{1}{p}} N_{\min}^{\frac1p} \prod_{i=1,2,3} \|v_i\|_{L_{t,x}^2},
\end{equation}
where $N_{\min} = \min_{i=1,2,3} N_i$ and $N_{\max}$, $M_{\min}$ are defined in a similar way.
\end{prop}
\begin{proof}
By an almost orthogonality in $L^2(\Sp^1)$ (see e.g.~\cite{Tao01}), we may assume that for each $i=1$, $2$, $3$, there exists a set $K_i \subset \Z$ whose cardinality is comparable to $M_{\min}$, i.e. $\# K_i \sim M_{\min}$, such that $v_i$ can be described by the Fourier series expansion as
\begin{equation}\label{est:FSE}
v_i(t,x) = v_i(t,r \cos \theta, r \sin \theta) = \sum_{k \in K_i} v_{i,k}(t,r) e^{ik\theta} \quad \mathrm{in} \ L^2(\R^3),
\end{equation}
where $v_{i,k}:= \langle v_i, e^{ik \theta} \rangle_{L_{\theta}^2}$. 

We recall that $\chi \in C^{\infty}_{0}((-2,2))$ is an even, non-negative function such that $\chi (t)=1$ for $|t|\leq 1$ and $\psi (t):=\chi (t)-\chi (2t)$. 
For $\mathfrak{N} \in 2^{\Z}$, we define $\phi_{\mathfrak{N}}(t):=\psi (\mathfrak{N}^{-1}t)$ and 
$\widehat{{\mathcal{P}}_{\mathfrak{N}}f}(\xi) := \phi_\mathfrak{N} (\xi) \widehat{f}(\xi)$. 
Then, for $f \in L^2(\R^2)$, clearly it holds that
\begin{equation}\label{equality:LP}
P_N = \mathcal{P}_{N} \quad \mathrm{if} \ N \in 2^{\N}, \qquad P_1 f = \sum_{\mathfrak{N} \leq 1} \mathcal{P}_{\mathfrak{N}} f.
\end{equation}
Here we recall $\widehat{P_1 f}(\xi)= \chi (\xi) \widehat{f}(\xi)$. 

Let $\rho \in C_0^{\infty}(\R)$ be a non-negative function satisfying $\supp \rho \subset (-2,2)$ such that
\[
\sum_{\ell \in \Z} \rho(\tau-\ell) = 1 \quad \mathrm{for} \ \mathrm{any} \ \tau \in \R.
\]
For $\ell \in \Z$, we put $\rho_{\ell}(\tau) = \rho(\tau-\ell)$. 
For $\ell \in \Z$ and $u \in L^2(\R \times \R^2)$, we define the operator $R_{\ell}$ as
\[
\F_{t,x} R_{\ell} u(\tau,\xi)= \rho_{\ell}(\tau-|\xi|^2) \F_{t,x}u(\tau,\xi).
\]
Let $L \in 2^{\N_0}$ and
\[
J(L) = \{\ell \in \Z\, | \, \supp \rho_{\ell} \cap \supp\psi_L \not= \emptyset\}. 
\]
Then, the cardinality of $J(L)$ is comparable to $L$ and for $u \in L^2(\R \times \R^2)$ it follows that
\begin{equation}\label{decomposition:Q_L}
Q_L u = \sum_{\ell \in J(L)} R_{\ell} Q_Lu.
\end{equation}

For $i=1,2,3$, let $\ell_i \in \Z$ and define $w_{i,\ell_i} = R_{\ell_i}\mathcal{P}_{{\mathfrak{N}}_i} v_i$. 
We will prove the following estimate.
\begin{equation}\label{est:prop4.4-01}
\Bigl| \int w_{1,\ell_1} \overline{w}_{2,\ell_2} \overline{w}_{3,\ell_3} dt dx \Bigr| \lesssim M_{\min}^{\frac12 - \frac1p}{\mathfrak{N}}_{\max}^{-1+\frac{1}{p}} {\mathfrak{N}}_{\min}^{\frac1p} \prod_{i=1,2,3} \|w_{i,\ell_i}\|_{L_{t,x}^2},
\end{equation}
Here, the implicit constant does not depend on $\ell_1$, $\ell_2$, and $\ell_3$.
Let us see that this estimate implies \eqref{est:goal-prop4.3}. 
First we consider the case $N_{\min} \geq 2$. 
In this case, \eqref{equality:LP} implies $w_{i,\ell_i} = R_{\ell_i} v_i$. 
It follows from \eqref{decomposition:Q_L}, \eqref{est:prop4.4-01} and the Cauchy-Schwarz inequality that
\begin{align*}
\Bigl| \int v_1 \overline{v}_2 \overline{v}_3 dt dx \Bigr| 
& \leq 
\sum_{i=1}^3
\sum_{\ell_i \in J(L_i)} \Bigl| \int w_{1,\ell_1} \overline{w}_{2,\ell_2} \overline{w}_{3,\ell_3} dt dx \Bigr|\\
& \lesssim 
\sum_{i=1}^3
\sum_{\ell_i \in J(L_i)} M_{\min}^{\frac12 - \frac1p}N_{\max}^{-1+\frac{1}{p}} N_{\min}^{\frac1p}  
 \|w_{1,\ell_1}\|_{L_{t,x}^2} \|w_{2,\ell_2}\|_{L_{t,x}^2} \|w_{3,\ell_3}\|_{L_{t,x}^2}\\
& \lesssim (L_1L_2L_3)^{\frac12} M_{\min}^{\frac12 - \frac1p}N_{\max}^{-1+\frac{1}{p}} N_{\min}^{\frac1p}  
\|v_1\|_{L_{t,x}^2} \|v_2\|_{L_{t,x}^2} \|v_3\|_{L_{t,x}^2}.
\end{align*}
In the last estimate, we utilized $\# J(L_i) \sim L_i$. 
The case $N_{\min} =1$ is handled in a similar way. 
Here we consider the case $N_{\min}=N_3=1$. 
The other cases $N_{\min} = N_1=1$ and $N_{\min}=N_2 =1$ can be treated in the same way. 
Clearly, we may assume ${{N}}_{\max} \geq 2$. 
By applying \eqref{equality:LP} and \eqref{decomposition:Q_L} to $v_3$ as
\[
v_3 = \sum_{\ell_3 \in J(L_3)}\sum_{\mathfrak{N}_3 \leq 2}  R_{\ell_3}\mathcal{P}_{{\mathfrak{N}}_3} v_3
=\sum_{\ell_3 \in J(L_3)}\sum_{\mathfrak{N}_3 \leq 2} w_{3,\ell_3},
\]
we may utilize \eqref{est:prop4.4-01} and obtain \eqref{est:goal-prop4.3}. 

Let us consider \eqref{est:prop4.4-01}. By Plancherel's theorem, it suffices to prove
\begin{equation}\label{est:prop4.4-01.5}
\begin{split}
&|(\F_{t,x}w_{1,\ell_1}) * (\F_{t,x}\overline{w}_{2,\ell_2}) *(\F_{t,x}\overline{w}_{3,\ell_3}) (0)|\\
& \lesssim 
M_{\min}^{\frac12 - \frac1p}N_{\max}^{-1+\frac{1}{p}} N_{\min}^{\frac1p} \prod_{i=1,2,3}\|\F_{t,x}w_{i,\ell_i}\|_{L_{\tau,\xi}^2},
\end{split}
\end{equation}
where the implicit constant does not depend on $\ell_1$, $\ell_2$, and $\ell_3$. 
Put
\[
f_{1}= \F_{t,x}w_{1,\ell_1}, \quad f_{2} = \F_{t,x}\overline{w}_{2,\ell_2}, \quad f_{3} = \F_{t,x}\overline{w}_{3,\ell_3}.
\]
Then, we observe that
\begin{align}
\label{condition:supportf_1}
\supp f_{1}
& \subset  \{ (\tau,\xi) \in \R^3 \, |\, |\tau - |\xi|^2 -\ell_1| < 2, \ |\xi| \sim \mathfrak{N}_1 \},\\
\label{condition:supportf_2}
\supp f_{2} 
& \subset  \{ (\tau,\xi) \in \R^3 \, |\, |\tau + |\xi|^2 +\ell_2| < 2, \ |\xi| \sim \mathfrak{N}_2 \},\\
\label{condition:supportf_3}
\supp f_{3} 
& \subset  \{ (\tau,\xi) \in \R^3 \, |\, |\tau + |\xi|^2 + \ell_3| < 2, \ |\xi| \sim \mathfrak{N}_3 \}.
\end{align}
For $p \geq 2$, by using Bernstein's inequality and \eqref{est:FSE} with $\# K_i \sim M_{\min}$, we get $\|w_{i,\ell_i}\|_{L_{t,r}^2 L_{\theta}^p} \lesssim M_{\min}^{\frac12 - \frac1p} \|w_{i,\ell_i}\|_{L_{t,x}^2}$ for each $i=1,2,3$. 
This implies $\|f_{i}\|_{L_{\tau,r}^2 L_{\theta}^p} \lesssim M_{\min}^{\frac12 - \frac1p} \|f_{i}\|_{L_{\tau,\xi}^2}$ for each $i=1,2,3$. 
To avoid redundancy, let us assume ${\mathfrak{N}}_3 \lesssim {\mathfrak{N}}_1 \sim {\mathfrak{N}}_2$. 
Then, \eqref{est:prop4.4-01.5} is verified by showing
\begin{equation}\label{est:prop4.4-02}
|f_{1} * f_{2} *f_{3} (0)| \lesssim 
{\mathfrak{N}}_{1}^{-1+\frac{1}{p}} {\mathfrak{N}}_{3}^{\frac1p} \|f_{1}\|_{L_{\tau,r}^2 L_{\theta}^p}\|f_{2}\|_{L_{\tau,\xi}^2} \|f_{3}\|_{L_{\tau,\xi}^2}.
\end{equation}
Put $g_{i} (\tau,\xi)= f_{i}({\mathfrak{N}}_1^2 \tau, {\mathfrak{N}}_1 \xi)$. Then, \eqref{condition:supportf_1}, \eqref{condition:supportf_2}, \eqref{condition:supportf_3} mean that
\begin{align*}
\supp g_{1}
& \subset  \Bigl\{ (\tau,\xi) \in \R^3 \, \Big|\, \Bigl|\tau - |\xi|^2 -\frac{\ell_1}{{\mathfrak{N}}_1^2} \Bigr| < \frac{2}{{\mathfrak{N}}_1^2}, \ |\xi| \sim 1 \Bigr\},\\
\supp g_{2} 
& \subset  \Bigl\{ (\tau,\xi) \in \R^3 \, \Big|\, \Bigl|\tau + |\xi|^2 +\frac{\ell_2}{{\mathfrak{N}}_1^2}\Bigr| < \frac{2}{{\mathfrak{N}}_1^2}, \ |\xi| \sim 1 \Bigr\},\\
\supp g_{3} 
& \subset  \Bigl\{ (\tau,\xi) \in \R^3 \, \Big|\, \Bigl|\tau + |\xi|^2 +\frac{\ell_3}{{\mathfrak{N}}_1^2}\Bigr| < \frac{2}{{\mathfrak{N}}_1^2}, \ |\xi| \sim \frac{\mathfrak{N}_3}{{\mathfrak{N}}_1} \Bigr\},
\end{align*}
and \eqref{est:prop4.4-02} is rewritten as
\[
|g_{1} * g_{2} * g_{3} (0)| \lesssim {\mathfrak{N}}_{1}^{-3+\frac{1}{p}} {\mathfrak{N}}_{3}^{\frac1p} \|g_{1}\|_{L_{\tau,r}^2 L_{\theta}^p}\|g_{2}\|_{L_{\tau,\xi}^2} \|g_{3}\|_{L_{\tau,\xi}^2}.
\]
This follows from Theorem~\ref{theorem:trilinearest} with $\varepsilon = {\mathfrak{N}}_1^{-2}$ and $N_3 = {\mathfrak{N}}_{3}/ {\mathfrak{N}}_{1}$, since the functions
\[
\varphi_1(\xi)= |\xi|^2 + \frac{\ell_1}{{\mathfrak{N}}_1^2}, \quad 
\varphi_2(\xi)= - |\xi|^2 - \frac{\ell_2}{{\mathfrak{N}}_1^2} \quad
\varphi_3(\xi)= -|\xi|^2 - \frac{\ell_3}{{\mathfrak{N}}_1^2}
\]
clearly satisfy \textit{Assumption}~\ref{assumption:hypersurface}.
\end{proof}
Next, by using Theorems~\ref{thm:GNLWR3}, we prove the second tilinear estimate.
\begin{prop}\label{prop4.5}
For $i=1,2,3$, let $L_i$, $N_i \in 2^{\N_0}$, and $u_i \in L^2(\R \times \R^2)$. We define
\[
v_{i} = Q_{L_i} P_{N_i} u_i.
\]
Then, we have
\begin{equation}\label{goal:prop4.5}
\Bigl| \int v_1 \overline{v}_2 \overline{v}_3 dt dx \Bigr| \lesssim (L_1L_2L_3)^{\frac12} (N_{\min} N_{\max})^{-\frac12} \prod_{i=1,2,3} \|v_i\|_{L_{t,x}^2}.
\end{equation}
\end{prop}
\begin{proof}
When $N_{\min}=1$, the claim follows from the following $L^2$-bilinear estimate:
\begin{align*}
& \|Q_{L_1}P_{N_1}u_1 Q_{L_2}P_{N_2} u_2 \|_{L_{t,x}^2} + \|Q_{L_1}P_{N_1}u_1 Q_{L_2}P_{N_2} \overline{u}_2 \|_{L_{t,x}^2}\\
& \lesssim (L_1 L_2)^{\frac12} \biggl( \frac{\min(N_1,N_2)}{\max(N_1,N_2)} \biggr)^{\frac12} \| Q_{L_1}P_{N_1}u_1\|_{L_{t,x}^2} \| Q_{L_2}P_{N_2} u_2\|_{L_{t,x}^2}.
\end{align*}
For the proof of this estimate, see e.g. Lemma 1 in \cite{CDKS01} and Lemma~3.1 in \cite{H14}.

Hereafter, we assume $N_{\min} \geq 2$. Suppose that $L_1\gtrsim N_{\min} N_{\max}$. By using Corollary~\ref{Bo_Stri}
\[
\|Q_{L}P_{N} u\|_{L_{t,x}^4} \lesssim L^{\frac12} \|Q_{L}P_{N} u\|_{L_{t,x}^2},
\]
we obtain
\begin{align*}
\Bigl| \int v_1 \overline{v}_2 \overline{v}_3 dt dx \Bigr| & 
\lesssim \|v_1\|_{L_{t,x}^2} \|v_2\|_{L_{t,x}^4}\| v_3 \|_{L_{t,x}^4}\\
&\lesssim (L_1L_2L_3)^{\frac12} (N_{\min} N_{\max})^{-\frac12} \prod_{i=1,2,3} \|v_i\|_{L_{t,x}^2}.
\end{align*}
The cases $L_2 \gtrsim N_{\min} N_{\max}$ and $L_3 \gtrsim N_{\min} N_{\max}$ can be treated in a similar way. Hence, we assume $L_{\max} = \max(L_1,L_2,L_3) \ll N_{\min} N_{\max}$. 

Let $J \subset \N$ and $\{\omega_j\}_{j \in J}$ be the collection of $\omega_j \in \Sp^{1}$ such that 
\[
2^{-10} \frac{N_{\min}}{N_{\max}} \leq \inf_{\substack{j_1,j_2 \in J\\
j_1 \not= j_2}} | \omega_{j_1} - \omega_{j_2} | \leq 2^{-8} \frac{N_{\min}}{N_{\max}}. 
\]
Hence, the cardinality of $J$ is comparable to $N_{\max}/N_{\min}$. 
We define
\begin{align*}
& \mathfrak{D}_{j} = \Bigl\{ r \omega \in \R^2 \, \Bigl| \, r>0, \, \omega \in \Sp^{1} \ \mathrm{such \ that} \  |\omega - \omega_j| \leq 2^{-10} \frac{N_{\min}}{N_{\max}} \Bigr\},\\
& \mathcal{I} = \Bigl\{(j_1,j_2) \in J \times J \, \Bigl| \, 2^{-9} \frac{N_{\min}}{N_{\max}} \leq \min_{\pm}|\omega_{j_1} \pm \omega_{j_2}| \leq 2^5 \frac{N_{\min}}{N_{\max}} \Bigr\},\\
& \widetilde{\mathcal{I}} = \Bigl\{(j_1,j_2) \in J \times J \, \Bigl| \,  \min_{\pm}|\omega_{j_1} \pm \omega_{j_2}| \leq 2^{-9} \frac{N_{\min}}{N_{\max}} \Bigr\}.
\end{align*}
Note that $\bigcup_{j \in J} \mathfrak{D}_{j} = \R^2 \setminus \{0\}$. 
By Plancherel's theorem, we can write
\[
\Bigl| \int v_1 \overline{v}_2 \overline{v}_3 dt dx \Bigr| = \Bigl| \int_* \F_{t,x}{v}_1(\tau_1,\xi_1) \overline{\F_{t,x}{v}_2}(\tau_2,\xi_2) \overline{\F_{t,x}{v}_3}(\tau_3,\xi_3) d\sigma \Bigr|,
\] 
where $*$ denotes $(\tau_3,\xi_3)=(\tau_1-\tau_2, \xi_1-\xi_2)$ and $d \sigma = d\tau_1 d\xi_1 d\tau_2 d\xi_2$. 
Let $1 \leq \ell_{\min}, \ell_{\mathrm{med}}, \ell_{\max} \leq 3$ satisfy $\{\ell_{\min}, \ell_{\mathrm{med}}, \ell_{\max}\} = \{1,2,3\}$ and $N_{\ell_{\min}} \leq N_{\ell_{\mathrm{med}}} \leq N_{\ell_{\max}}$. 
In the above, if $(\xi_{\ell_{\mathrm{med}}}, \xi_{\ell_{\max}}) \in \mathfrak{D}_{j_1} \times \mathfrak{D}_{j_2}$ with $(j_1,j_2) \in \widetilde{\mathcal{I}}$ then the following holds:
\begin{equation}\label{est:prop4.5-modulation}
L_{\max} \gtrsim N_{\min}  N_{\max},
\end{equation}
which contradicts the assumption $L_{\max} \ll N_{\min}N_{\max} $. 
The inequality \eqref{est:prop4.5-modulation} can be seen by a simple computation. For instance, if $2 \leq N_3 \leq N_2 \leq N_1$, since we may assume $N_1 \sim N_2$, we have
\begin{align*}
L_{\max} & \gtrsim \bigl| (\tau_1-|\xi_1|^2) - (\tau_2 - |\xi_2|^2) -(\tau_3 - |\xi_3|^2) \bigr|\\
& =||\xi_1|^2 - |\xi_2|^2 - |\xi_3|^2| = 2|\xi_2 \cdot (\xi_1-\xi_2)|\\
& = 2  |\xi_2| |\xi_1-\xi_2| \bigl| \cos \angle(\xi_2, \xi_1 -\xi_2) \bigr|  \gtrsim N_1 N_3.
\end{align*}
The last inequality follows from $(\xi_1, \xi_2) \in \mathfrak{D}_{j_1} \times \mathfrak{D}_{j_2}$ with $(j_1,j_2) \in \widetilde{\mathcal{I}}$. 
The other cases can be treated in a similar way. Consequently, we can assume $(\xi_{\ell_{\mathrm{med}}},\xi_{\ell_{\max}}) \in \bigcup_{(j_1,j_2)\in \mathcal{I}}\mathfrak{D}_{j_1} \times \mathfrak{D}_{j_2}$. 

To avoid redundancy, we assume $N_3 \leq N_2 \leq N_1$ here. 
Note that $j_1$ and $j_2$ are almost one-to-one correspondence. 
We invoke the same reduction from~\eqref{est:goal-prop4.3} to~\eqref{est:prop4.4-02} in the proof of Proposition~\ref{prop4.4}. 
Let $2 \leq N_3 \leq N_2 \leq  N_1$, $(j_1,j_2) \in \mathcal{I}$, $\ell_1$, $\ell_2$, $\ell_3 \in \Z$ and functions $f_{1}$, $f_{2}$, $f_{3} \in L^2(\R^3)$ satisfy
\begin{align*}
\supp f_{1}
& \subset  \{ (\tau,\xi) \in \R^3 \, |\, |\tau - |\xi|^2 -\ell_1| < 2, \ |\xi| \sim N_1 \}\cap \mathfrak{D}_{j_1},\\
\supp f_{2} 
& \subset  \{ (\tau,\xi) \in \R^3 \, |\, |\tau + |\xi|^2 +\ell_2| < 2, \ |\xi| \sim N_2 \}\cap \mathfrak{D}_{j_2},\\
\supp f_{3} 
& \subset  \{ (\tau,\xi) \in \R^3 \, |\, |\tau - |\xi|^2 - \ell_3| < 2, \ |\xi| \sim N_3 \}. 
\end{align*}
The estimate \eqref{goal:prop4.5} follows from
\begin{equation}\label{est:prop4.5-01}
\begin{split}
&\Bigl| \int   \bigl(f_{1} * f_{2}\bigr) f_{3} d\tau d\xi \Bigr|\\
& \lesssim (N_1N_3)^{-\frac12}\|f_{1}\|_{L_{\tau,\xi}^2} \|f_{2}\|_{L_{\tau,\xi}^2} \|f_{3}\|_{L_{\tau,\xi}^2}.
\end{split}
\end{equation}
Here, the implicit function does not depend on $\ell_1$, $\ell_2$, and $\ell_3$. 
As in the proof of Proposition~\ref{prop4.4}, we put
\begin{align*}
g_{1} (\tau,\xi)& = f_{1}(N_1^2 \tau, N_1 \xi), \\
g_{2} (\tau,\xi)& = f_{2}(N_1^2 \tau, N_1 \xi), \\
g_{3} (\tau,\xi)& = f_{3}(N_1^2 \tau, N_1 \xi),
\end{align*}
and rewrite \eqref{est:prop4.5-01} as
\begin{equation}\label{est:prop4.5-02}
\begin{split}
&\Bigl| \int   \bigl(g_{1} * g_{2}\bigr) g_{3} d\tau d\xi \Bigr|\\
& \lesssim N_{1}^{-\frac{5}{2}}N_3^{-\frac12} \|g_{1}\|_{L_{\tau,\xi}^2}\|g_{2}\|_{L_{\tau,\xi}^2} \|g_{3}\|_{L_{\tau,\xi}^2}.
\end{split}
\end{equation}
Here it holds that 
\begin{align*}
\supp g_{1}
& \subset  \Bigl\{ (\tau,\xi) \in \R^3 \, \Big|\, \Bigl|\tau - |\xi|^2 -\frac{\ell_1}{N_1^2} \Bigr| < \frac{2}{N_1^2}, \ |\xi| \sim 1 \Bigr\}\cap \mathfrak{D}_{j_1},\\
\supp g_{2} 
& \subset  \Bigl\{ (\tau,\xi) \in \R^3 \, \Big|\, \Bigl|\tau + |\xi|^2 +\frac{\ell_2}{N_1^2} \Bigr| < \frac{2}{N_1^2}, \ |\xi| \sim 1 \Bigr\} \cap \mathfrak{D}_{j_2},\\
\supp g_{3} 
& \subset  \Bigl\{ (\tau,\xi) \in \R^3 \, \Big|\, \Bigl|\tau - |\xi|^2 -\frac{\ell_3}{N_1^2} \Bigr| < \frac{2}{N_1^2}, \ |\xi| \sim \frac{N_3}{N_1} \Bigr\}.
\end{align*}
Further, after a harmless decomposition, we may assume that there exists $\xi' \in \R^2$ such that $|\xi'| \sim N_3/N_1$ and
\[
\supp g_{3} \subset \Bigl\{ (\tau,\xi) \in \R^3 \, \Big|\, \Bigl|\tau - |\xi|^2 -\frac{\ell_3}{N_1^2} \Bigr| < \frac{2}{N_1^2}, \ |\xi-\xi'| \leq 2^{-100} \frac{N_3}{N_1} \Bigr\}.
\]
To prove \eqref{est:prop4.5-02}, we utilize \eqref{est:ThickenedLW} with $\varepsilon = N_1^{-2}$ and $A \sim N_1/N_3$. 

Let us set 
\[
\varphi_1(\xi)= |\xi|^2 + \frac{\ell_1}{N_1^2}, \quad 
\varphi_2(\xi)= - |\xi|^2 - \frac{\ell_2}{N_1^2}, \quad
\varphi_3(\xi)= |\xi|^2 + \frac{\ell_3}{N_1^2},
\] 
and define 
\begin{align*}
&S_1 = \{(\varphi_1(\xi),\xi) \in \R^3 \, | \, |\xi| \sim 1 , \ \xi \in \mathfrak{D}_{j_1}\},\\
&S_2 = \{(\varphi_2(\xi),\xi) \in \R^3 \, | \, |\xi| \sim 1 , \ \xi \in \mathfrak{D}_{j_2}\},\\
& S_3 = \Bigl\{(\varphi_3(\xi),\xi) \in \R^3 \, | \,  |\xi-\xi'| \leq 2^{-100} \frac{N_3}{N_1} \Bigr\}.
\end{align*}
Clearly, because $\varphi_i \in C^{\infty}(\R^2)$, $S_1$, $S_2$, $S_3$ are $C^{1,1}$ hypersurfaces. 
We check that the transversality condition \eqref{eq:TransversalityAssumption} in \textit{Assumption}~\ref{AssumptionSurfaces} holds with $A \sim N_1/N_3$. 
We can describe the unit normal of $S_{i}$ on $(\varphi_i(\xi),\xi) = (\varphi_i(\xi), \xi^{(1)}, \xi^{(2)})$ which we write $\mathfrak{n}_{i}(\varphi_i(\xi),\xi)$ as
\begin{align*}
&\mathfrak{n}_{1}(\varphi_1(\xi),\xi) = \frac{1}{\sqrt{1+4 |\xi|^2}}(- 1,  2 \xi^{(1)}, 2 \xi^{(2)}),\\
&\mathfrak{n}_{2}(\varphi_2(\xi),\xi) = \frac{1}{\sqrt{1+4 |\xi|^2}}( 1,  2 \xi^{(1)}, 2 \xi^{(2)}),\\
&\mathfrak{n}_{3}(\varphi_3(\xi),\xi) = \frac{1}{\sqrt{1+4 |\xi|^2}}(- 1,  2 \xi^{(1)}, 2 \xi^{(2)}).
\end{align*}
Let $|\xi_1| \sim |\xi_2| \sim 1$, $(\xi_1, \xi_2) \in \mathfrak{D}_{j_1} \times \mathfrak{D}_{j_2}$, $|\xi_1+\xi_2-\xi'| \leq 2^{-100}N_3/N_1$, and $|\xi_3-\xi'| \leq 2^{-100}N_3/N_1$. Then we compute that
\begin{align*}
&\bigl|\mathrm{det}\bigl( \mathfrak{n}_{1}(\varphi_1(\xi_1),\xi_1), \mathfrak{n}_{2}(\varphi_2(\xi_2),\xi_2),\mathfrak{n}_{3}(\varphi_3(\xi_3),\xi_3)\bigr) \bigr|\\
&\gtrsim \left|\textnormal{det}
\begin{pmatrix}
- 1 &  1 &  - 1 \\
\xi_1^{(1)}  & \xi_2^{(1)} & \xi_3^{(1)} \\
\xi_1^{(2)}  & \xi_2^{(2)} & \xi_3^{(2)}
\end{pmatrix} \right|  \\
&\gtrsim\left|\textnormal{det}
\begin{pmatrix}
- 1 &  1 &  - 1 \\
\xi_1^{(1)}  & \xi_2^{(1)} & \xi_1^{(1)} +\xi_2^{(1)} \\
\xi_1^{(2)}  & \xi_2^{(2)} & \xi_1^{(2)} +\xi_2^{(2)}
\end{pmatrix} \right| - 2^{-20}\frac{N_3}{N_1}\\
& \geq  |\xi_1^{(1)} \xi_2^{(2)} - \xi_1^{(2)} \xi_2^{(1)}|- 2^{-20}\frac{N_3}{N_1}\\
& = |\xi_1| |\xi_2| | \sin \angle (\xi_1,\xi_2)|- 2^{-20}\frac{N_3}{N_1} \gtrsim \frac{N_3}{N_1}.
\end{align*}
Thus, the transversality condition \ref{eq:TransversalityAssumption} in \textit{Assumption}~\ref{AssumptionSurfaces} holds with $A \sim N_1/N_3$, and therefore \eqref{est:prop4.5-02} follows from \eqref{est:ThickenedLW}.
\end{proof}
We now prove Proposition~\ref{prop2.2}.
\begin{proof}[Proof of Proposition~\ref{prop2.2}]
Let $N_1$, $N_2$, $N_3 \in 2^{\N_0}$. Our goal is to prove that there exists small $\kappa = \kappa(s) >0$ such that
\begin{equation}\label{est:prop4.1-01}
\begin{split}
\bigl\|P_{N_3}\F_{t,x}^{-1} \bigl( \langle \tau - |\xi|^2 \rangle^{-1} \F_{t,x}({P_{N_1}u}_1 {\overline{P_{N_2}u_2}})\bigr) \bigr\|_{Z^{s,\sigma}_{\kappa}}& \\
 \lesssim 
N_{\min}^{-\kappa}\|P_{N_1}u_1\|_{Z^{s,\sigma}_{\kappa}} \|P_{N_2}u_2\|_{Z^{s,\sigma}_{\kappa}}&.
\end{split}
\end{equation}
For some simple cases, we show the following estimate instead of \eqref{est:prop4.1-01}. 
\begin{equation}\label{est:prop4.1-02}
\|P_{N_3}({P_{N_1}u}_1{\overline{P_{N_2}u_2}})\bigr\|_{X^{s,\sigma; -\frac12 + \kappa}} \lesssim 
N_{\min}^{-\kappa}\|P_{N_1}u_1\|_{X^{s,\sigma;\frac12-\kappa}} \|P_{N_2}u_2\|_{X^{s,\sigma;\frac12-\kappa}}.
\end{equation}
Since $\|u\|_{Z^{s,\sigma}_{\kappa}} \leq \|u\|_{X^{s, \sigma;\frac12 + \kappa}}$ and $\|u\|_{X^{s, \sigma;\frac12 - \kappa}} \leq \|u\|_{Z^{s,\sigma}_{\kappa}}$, \eqref{est:prop4.1-02} implies \eqref{est:prop4.1-01}. 
We will show \eqref{est:prop4.1-02} under the condition $\sigma =0$. 
The case $\sigma >0$ is handled in the same way. By duality, \eqref{est:prop4.1-02} with $\sigma=0$ is given by
\begin{equation}\label{est:prop4.1-03}
\Bigl| \int v_1 \overline{v}_2 \overline{v}_3 dt dx \Bigr| \lesssim (L_1L_2L_3)^{\frac12-2 \kappa} N_{\min}^{-\kappa} N_1^{s} N_2^s N_3^{-s} \prod_{i=1,2,3} \|v_i\|_{L_{t,x}^2},
\end{equation}
where $v_{i} = Q_{L_i} P_{N_i} u_i$ $(i=1,2,3)$. First, suppose that $L_{\max} \gtrsim N_{\max}^2$. If $L_1 =L_{\max}$ then by using Corollary~\ref{Bo_Stri}, we have
\begin{align*}
\Bigl| \int v_1 \overline{v}_2 \overline{v}_3 dt dx \Bigr| & \leq
\|v_1\|_{L_{t,x}^2} \|v_2\|_{L_{t,x}^4} \|v_3\|_{L_{t,x}^4}\\
& \lesssim N_{\max}^{-1 + 12 \kappa} L_1^{\frac12 - 6 \kappa} (L_2 L_3)^{\frac12}\prod_{i=1,2,3} \|v_i\|_{L_{t,x}^2}\\
& \leq N_{\max}^{-1 + 12 \kappa} (L_1L_2 L_3)^{\frac12- 2 \kappa}\prod_{i=1,2,3} \|v_i\|_{L_{t,x}^2}.
\end{align*}
Since $s > -1/2$ and $\kappa$ is sufficiently small, this gives \eqref{est:prop4.1-03}. 
The other cases $L_{2} = L_{\max} \gtrsim N_{\max}^2$ and $L_{3} = L_{\max} \gtrsim N_{\max}^2$ can be handled in the same way. 
Next, we assume $L_{\max} \ll N_{\max}^2$ and $N_3 \sim N_{\max}$. Then, we use Proposition~\ref{prop4.5} and obtain
\begin{align*}
\Bigl| \int v_1 \overline{v}_2 \overline{v}_3 dt dx \Bigr| & \lesssim (L_1L_2 L_3)^{\frac12} (N_{\min} N_{\max})^{-\frac12} \prod_{i=1,2,3} \|v_i\|_{L_{t,x}^2}\\
& \lesssim (L_1L_2 L_3)^{\frac12-2 \kappa}N_{\min}^{-\frac12} N_{\max}^{-\frac12 + 12 \kappa}\prod_{i=1,2,3} \|v_i\|_{L_{t,x}^2},
\end{align*}
which gives \eqref{est:prop4.1-03}.

Now we consider \eqref{est:prop4.1-01}. From the above observation, we can assume $N_3 \ll N_1 \sim N_2$ and replace $u_1$, $u_2$ by $Q_{L_1 \ll N_1^2}u_1$, $Q_{L_2 \ll N_1^2} u_2$, respectively. Here let us assume $s<-1/4$. Later, we will give a small comment on the case $s=-1/4$ and $\sigma =0$ since a slight modification is required. Recall that, since \eqref{est:Znorm}, it holds that
\[
\|Q_{L \ll N^2} P_{N} u\|_{Z^{s,\sigma}_{\kappa}} \sim 
\|Q_{L \ll N^2} P_{N} u\|_{X^{s, \sigma;\frac12 + \kappa}}.
\]
Thus, letting $v_{j} = Q_{L_j}H_{M_j} P_{N_j} u_j$ with the condition $L_j \ll N_j^2$ $(j=1,2)$, we will prove
\begin{equation}\label{est:prop4.1-04}
\begin{split}
& \bigl\|H_{M_3} P_{N_3}\F_{t,x}^{-1} \bigl(\langle \tau - |\xi|^2 \rangle^{-1} \F_{t,x}({v}_1  {\overline{v}_2})\bigr)\bigr\|_{Z_{\kappa}^{0,0}}\\
& \lesssim  (L_1L_2)^{\frac12} \min\Bigl(\Bigl(\frac{M_{\min}}{N_1}\Bigr)^{\delta}, \Bigl( \frac{N_1}{M_{\min}}\Bigr)^{\delta} \Bigr)M_{\min}^{- 2 s -\frac12}N_1^{2 s} N_3^{-s -\kappa}
\|v_1\|_{L_{t,x}^2}\|v_2\|_{L_{t,x}^2}
\end{split}
\end{equation}
for some $\delta=\delta(s)>0$. Clearly \eqref{est:prop4.1-01} is implied by \eqref{est:prop4.1-04}. 
Recall that $Z^{s,\sigma}_{\kappa}$ norm is described as \eqref{est:Znorm}. 
We decompose $v_1 \overline{v}_2$ as $v_1 \overline{v}_2 = Q_{L_3 \ll N_3^2} (v_1 \overline{v}_2)+ Q_{L_3 \gtrsim N_3^2} (v_1 \overline{v}_2)$. 
The necessary estimate for $Q_{L_3 \ll N_3^2} (v_1 \overline{v}_2)$ is the following:
\begin{equation}\label{est:prop4.1-04.1}
\begin{split}
& \|Q_{L_3 \ll N_3^2} H_{M_3}P_{N_3}(v_1 \overline{v}_2)\|_{X^{0, 0;-\frac12+\kappa}}\\
& \lesssim  (L_1L_2)^{\frac12} \min\Bigl(\Bigl(\frac{M_{\min}}{N_1}\Bigr)^{\delta}, \Bigl( \frac{N_1}{M_{\min}}\Bigr)^{\delta} \Bigr)M_{\min}^{- 2 s -\frac12}N_1^{2 s} N_3^{-s -\kappa}
\|v_1\|_{L_{t,x}^2}\|v_2\|_{L_{t,x}^2}.
\end{split}
\end{equation}
Notice that
\begin{equation}
\label{est:prop4.1-04.5}
\|Q_{L_3 \ll N_3^2} H_{M_3}P_{N_3}(v_1 \overline{v}_2)\|_{X^{0, 0;-\frac12+\kappa}} 
\leq N_3^{4 \kappa} \|Q_{L_3 \ll N_3^2} H_{M_3}P_{N_3}(v_1 \overline{v}_2)\|_{X^{0, 0;-\frac12-\kappa}}.
\end{equation}
Therefore, to show \eqref{est:prop4.1-04.1}, it is enough to prove the following two estimates:
\begin{align}
\label{est:prop4.1-04.6}
& \|Q_{L_3 \ll N_3^2} H_{M_3}P_{N_3}(v_1 \overline{v}_2)\|_{X^{0, 0;-\frac12-\kappa}} \lesssim  (L_1L_2)^{\frac12}(N_1 N_3)^{-\frac12}
\|v_1\|_{L_{t,x}^2}\|v_2\|_{L_{t,x}^2},\\
\label{est:prop4.1-04.7}
& \|Q_{L_3 \ll N_3^2} H_{M_3}P_{N_3}(v_1 \overline{v}_2)\|_{X^{0, 0;-\frac12-\kappa}} \lesssim  (L_1L_2)^{\frac12}M_{\min}^{\frac12-\frac1p}N_1^{-1+\frac1p} N_3^{\frac1p}
\|v_1\|_{L_{t,x}^2}\|v_2\|_{L_{t,x}^2},
\end{align}
where $2 \leq p <\infty$ in the second estimate. 
Indeed, since $-1/2 < s < -1/4$ and $\delta$, $\kappa >0$ are sufficiently small, \eqref{est:prop4.1-04.5}, \eqref{est:prop4.1-04.6} establish \eqref{est:prop4.1-04.1} in the case $N_1 \leq M_{\min}$. 
On the other hand, in the case $N_1 \geq M_{\min}$, by taking $p$ as $0 < 1/p < \min(1 + 2 s -\delta, -s - 5 \kappa)$, the two estimates \eqref{est:prop4.1-04.5}, \eqref{est:prop4.1-04.7} yield \eqref{est:prop4.1-04.1}. 
Note that, by duality, \eqref{est:prop4.1-04.6} and \eqref{est:prop4.1-04.7} follow from Propositions~\ref{prop4.5} and \ref{prop4.4}, respectively, i.e.
\begin{align}\label{est:prop4.1-05}
& \Bigl| \int v_1 \overline{v}_2 \overline{v}_3 dt dx \Bigr| \lesssim (L_1L_2L_3)^{\frac12} (N_1 N_3)^{-\frac12} \prod_{i=1,2,3} \|v_i\|_{L_{t,x}^2},\\
\label{est:prop4.1-06}
&\Bigl| \int v_1 \overline{v}_2 \overline{v}_3 dt dx \Bigr| \lesssim (L_1L_2L_3)^{\frac12} M_{\min}^{\frac12 - \frac1p}N_{1}^{-1+\frac{1}{p}} N_{3}^{\frac1p} \prod_{i=1,2,3} \|v_i\|_{L_{t,x}^2},
\end{align}
where $v_3 = Q_{L_3}H_{M_3}P_{N_3} u_3$. Thus, the proof of \eqref{est:prop4.1-04.1} is completed.

To handle $Q_{L_3 \gtrsim N_3^2} (v_1 \overline{v}_2)$, because of \eqref{est:Znorm}, it suffices to prove
\begin{align*}
& \|Q_{L_3 \gtrsim N_3^2} H_{M_3} P_{N_3}(v_1  \overline{v}_2)  \|_{X^{4 \kappa, 0;-\frac12 - \kappa}}\\
& +\|\F_{t,x}^{-1}\bigl( \langle \tau - |\xi|^2 \rangle^{-1}\F_{t,x}(Q_{L_3 \gtrsim N_3^2} H_{M_3}P_{N_3}(v_1 \overline{v}_2))\bigr)\|_{Y^{0,0}}\\
& \lesssim (L_1L_2)^{\frac12} \min\Bigl(\Bigl(\frac{M_{\min}}{N_1}\Bigr)^{\delta}, \Bigl( \frac{N_1}{M_{\min}}\Bigr)^{\delta} \Bigr)M_{\min}^{- 2 s -\frac12}N_1^{2 s} N_3^{-s -\kappa}
\|v_1\|_{L_{t,x}^2}\|v_2\|_{L_{t,x}^2}.
\end{align*}
Thanks to $-\kappa$, in the same way as above, the first term can be dealt with by \eqref{est:prop4.1-05} and \eqref{est:prop4.1-06}. 
We estimate the latter term. 
We deduce from the dual estimate and the above observation that it suffices to show that there exists $\delta'=\delta'(s)>0$ such that
\begin{align}\label{est:prop4.1-07}
&\Bigl| \int v_1 \overline{v}_2 \overline{v}_3 dt dx \Bigr|\lesssim (L_1L_2)^{\frac12} L_3^{1-\delta'}N_1^{-\frac12} \| v_1\|_{L_{t,x}^2} \|v_2\|_{L_{t,x}^2} \|\F_{t,x}{v}_3\|_{L_{\xi}^2 L_{\tau}^{\infty}},\\
\label{est:prop4.1-08}
\begin{split}
&\Bigl| \int v_1 \overline{v}_2 \overline{v}_3 dt dx \Bigr| \\ 
&\lesssim (L_1L_2)^{\frac12} L_3^{1-\delta'}
M_{\min}^{\frac12 - \frac1p}N_{1}^{-1+\frac{1}{p}} N_{3}^{\frac1p} \| v_1\|_{L_{t,x}^2} \|v_2\|_{L_{t,x}^2} \|\F_{t,x}{v}_3\|_{L_{\xi}^2 L_{\tau}^{\infty}},
\end{split}
\end{align}
where $v_3 = Q_{L_3}P_{N_3} u_3$ and $2 \leq p <\infty$ in the second estimate. 
To see \eqref{est:prop4.1-07} and \eqref{est:prop4.1-08}, 
by using \eqref{est:prop4.1-05}, \eqref{est:prop4.1-06}, and the Cauchy-Schwarz inequality, we get
\begin{align}\label{est:prop4.1-09}
& \Bigl| \int v_1 \overline{v}_2 \overline{v}_3 dt dx \Bigr| \lesssim (L_1L_2)^{\frac12} L_3 (N_1 N_3)^{-\frac12}\| v_1\|_{L_{t,x}^2} \|v_2\|_{L_{t,x}^2} \|\F_{t,x}{v}_3\|_{L_{\xi}^2 L_{\tau}^{\infty}},\\
\label{est:prop4.1-10}
& \Bigl| \int v_1 \overline{v}_2 \overline{v}_3 dt dx \Bigr| \lesssim (L_1L_2)^{\frac12} L_3 M_{\min}^{\frac12 - \frac1p}N_{1}^{-1+\frac{1}{p}} N_{3}^{\frac1p}\| v_1\|_{L_{t,x}^2} \|v_2\|_{L_{t,x}^2} \|\F_{t,x}{v}_3\|_{L_{\xi}^2 L_{\tau}^{\infty}},
\end{align}
for any $2 \leq p < \infty$. Here we used the support condition $\supp \F_{t,x}{v}_3 \subset \{|\tau-|\xi|^2|\leq L_3 \}$. 
We easily check that if the estimate
\begin{equation}\label{est:prop4.1-11}
\Bigl| \int v_1 \overline{v}_2 \overline{v}_3 dt dx \Bigr| \lesssim  (L_1L_2L_3)^{\frac12} N_1^{-\frac12} N_3^{\frac12} \| v_1\|_{L_{t,x}^2} \|v_2\|_{L_{t,x}^2} \|\F_{t,x}{v}_3\|_{L_{\xi}^2 L_{\tau}^{\infty}}
\end{equation}
holds, then by combining with \eqref{est:prop4.1-09} and \eqref{est:prop4.1-10}, we then have the desired bounds \eqref{est:prop4.1-07} and \eqref{est:prop4.1-08}. 

Let us consider \eqref{est:prop4.1-11}. A simple computation yields
\begin{align*}
\Bigl| \int v_1 \overline{v}_2 \overline{v}_3 dt dx \Bigr| & = 
\Bigl| \int \F_{t,x} \bigl( v_1  \overline{v}_3 \bigr) \overline{ \F_{t,x}{v}}_2 d\tau  d \xi \Bigr|\\
& \lesssim \Bigl\| \int \F_{t,x}{v}_1 (\tau + \tau_3,\xi+\xi_3) \overline{\F_{t,x}{v}_3}(\tau_3,\xi_3) d\tau_3 d\xi_3 \Bigr\|_{L_{\xi}^2 L_{\tau}^{\infty}} \|\F_{t,x}{v}_2\|_{L_{\xi}^2 L_{\tau}^1}\\
& \lesssim L_2^{\frac12} \Bigl\| \int \F_{t,x}{v}_1 (\tau + \tau_3,\xi+\xi_3) \overline{\F_{t,x}{v}_3}(\tau_3,\xi_3) d\tau_3 d\xi_3 \Bigr\|_{L_{\xi}^2 L_{\tau}^{\infty}} \|v_2\|_{L_{t,x}^2}.
\end{align*}
Thus, it suffices to show
\[
\begin{split}
\Bigl\| \int \F_{t,x}{v}_1 (\tau + \tau_3,\xi+\xi_3) & \overline{\F_{t,x}{v}_3}(\tau_3,\xi_3) d\tau_3 d\xi_3 \Bigr\|_{L_{\xi}^2 L_{\tau}^{\infty}} \\
 &\lesssim  (L_1 L_3)^{\frac12} N_1^{-\frac12} N_3^{\frac12}\|\F_{t,x}{v}_1\|_{L_{\tau, \xi}^2} \|\F_{t,x}{v}_3\|_{L_{\xi}^2 L_{\tau}^{\infty}}.
\end{split}
\]
The support conditions of $\F_{t,x}{v}_1$ and $\F_{t,x}{v}_3$ imply
\begin{align*}
\supp \F_{t,x}{v}_1 \subset \{(\tau,\xi) \in \R \times \R^2 \, | \, |\xi| \lesssim N_1, \ |\tau-|\xi|^2| \lesssim L_1\},\\
\supp \F_{t,x}{v}_3 \subset \{(\tau,\xi) \in \R \times \R^2 \, | \, |\xi| \lesssim N_3, \ |\tau-|\xi|^2| \lesssim L_3\}.
\end{align*}
Therefore, by H\"{o}lder's inequality, we have
\begin{align*}
& \Bigl\| \int \F_{t,x}{v}_1 (\tau + \tau_3,\xi+\xi_3) \overline{\F_{t,x}{v}_3}(\tau_3,\xi_3) d\tau_3 d\xi_3 \Bigr\|_{L_{\xi}^2 L_{\tau}^{\infty}}\\ & \lesssim 
\Bigl\| \Bigl( \int \bigl| \F_{t,x}{v}_1 (\tau + \tau_3,\xi+\xi_3)\bigr|^2 \bigl| \overline{\F_{t,x}{v}_3}(\tau_3,\xi_3)\bigr|^2 d\tau_3 d\xi_3\Bigr)^{\frac12} |E(\tau,\xi)|^{\frac12} \Bigr\|_{L_{\xi}^2 L_{\tau}^{\infty}}\\
& \lesssim \sup_{\tau,\xi: |\xi| \sim N_1} |E(\tau,\xi)|^{\frac12} \|\F_{t,x}{v}_1\|_{L_{\tau, \xi}^2} \|\F_{t,x}{v}_3\|_{L_{\xi}^2 L_{\tau}^{\infty}}.
\end{align*}
Here we used the assumption $N_3 \ll N_1$ to deduce $|\xi| \sim N_1$ and the set $E(\tau,\xi) \subset \R^{3}$ is defined by
\[
E(\tau,\xi) = \{(\tau_3,\xi_3) \in \R \times \R^2 \,| \, |\xi_3| \lesssim N_3, \ |\tau_3 - |\xi_3|^2| \lesssim L_3, \ |\tau+\tau_3 - |\xi+\xi_3|^2| \lesssim L_1 \}.
\]
We need to prove
\begin{equation}\label{est:prop4.1-12}
\sup_{\tau,\xi: |\xi| \sim N_1} |E(\tau,\xi)| \lesssim N_1^{-1}N_3 L_1 L_3.
\end{equation}
By rotation, without loss of generality, we may assume that $\xi \in \R^2$ is on the right half of the first axis, i.e. $\xi=(|\xi|,0)$. We use the notation $\xi_3 = (\xi_3^{(1)}, \xi_3^{(2)})$. 
For fixed $\xi_3$, we get
\begin{equation}\label{est:prop4.1-13}
\sup_{\tau,\xi: |\xi| \sim N_1} \bigl| \{ \tau_3 \, | \, (\tau_3, \xi_3) \in E(\tau, \xi) \}\bigr|
\lesssim \min(L_1, L_3).
\end{equation}
We observe that
\begin{align*}
\max (L_1, L_3) &
\gtrsim \bigl| (\tau+ \tau_3) - |\xi + \xi_3|^2 - (\tau_3 -|\xi_3|^2) \bigr|\\
& = \bigl| (\tau- |\xi|^2)- 2 \xi \cdot \xi_3 \bigr|
\end{align*}
and $ |\partial_{\xi_3^{(1)}} (\xi \cdot \xi_3)| =|\xi| \sim N_1$. Then, for fixed $\xi_3^{(2)}$, we have
\begin{equation}\label{est:prop4.1-14}
\sup_{\tau,\xi: |\xi| \sim N_1} \bigl| \{ \xi_3^{(1)} \, | \, (\tau_3, \xi_3) \in E(\tau, \xi) \}\bigr|
\lesssim N_1^{-1}\max(L_1, L_3).
\end{equation}
By combining \eqref{est:prop4.1-13}, \eqref{est:prop4.1-14} with the support condition $|\xi_3| \lesssim N_3$, we get \eqref{est:prop4.1-12}.

Lastly, we consider \eqref{est:prop4.1-01} in the case $s=-1/4$ and $\sigma=0$. Similarly to the $s<-1/4$ case, it suffices to see
\begin{equation}\label{est:prop4.1-15}
\bigl\|P_{N_3}\F_{t,x}^{-1} \bigl( \langle \tau - |\xi|^2 \rangle^{-1} \F_{t,x}({v}_1 {\overline{v}_2}) \bigr)\bigr\|_{Z_{\kappa}^{0,0}} \lesssim  (L_1L_2)^{\frac12} N_1^{-\frac12} 
\|v_1\|_{L_{t,x}^2}\|v_2\|_{L_{t,x}^2},
\end{equation}
where $N_3 \lesssim N_1 \sim N_2$, $v_j=Q_{L_j}P_{N_j}u_j$ $(j=1,2)$. 
We decompose $v_1 \overline{v}_2$ as $Q_{L_3 \ll N_3^2}(v_1 \overline{v}_2) + Q_{L_3 \gtrsim N_3^2} (v_1 \overline{v}_2)$. For the former term, \eqref{est:prop4.1-05} implies \eqref{est:prop4.1-15}. 
The latter term is handled by \eqref{est:prop4.1-09} and \eqref{est:prop4.1-11} in a similar way to the case $s<-1/4$.
\end{proof}
\section{On the optimalities of our results}\label{Section5}
In this section, we prove Theorems \ref{thm2}, \ref{thm3}, 
and Propositions~\ref{proposition2.2}, \ref{prop:SharpnessFactor}. First, we consider Theorem \ref{thm2}. The following argument is based on the proof of Theorem 1.2. (i) in \cite{KT10} by Kishimoto and Tsugawa.
\begin{proof}[Proof of Theorem \ref{thm2}]
Suppose that $s < -1/4$, $0 \leq \sigma < \min(-2 s -1/2, 1/2)$, and put $s_{\sigma} = -\sigma/2 -1/4$ which is equivalent to $\sigma = -2 s_{\sigma} -1/2$. Note that $\sigma < \min(-2 s -1/2, 1/2)$ means $\max(s, -1/2) < s_{\sigma}$. Let $N \gg 1$. We define the set $C_N \subset \R^2$ by
\[
C_N = \{ \xi=(|\xi| \cos \theta, |\xi| \sin \theta) \in \R^2 \, | \, N \leq |\xi| \leq N + N^{-1}, \  |\theta| \leq N^{-1}\}.
\]
Let $0<\delta \ll 1$. We define the function $f_{N,\delta}$ by
\[
\F_{x} f_{N,\delta} = \delta N^{\frac12 - s_{\sigma} - \sigma} \mathbf{1}_{C_N}.
\]
For $\ell \in \Z$, $r >0$, and $f \in L^2(\R^2)$, we define
\[
(f)_{\ell}(r) = \frac{1}{2 \pi} \int_0^{2 \pi} f(r \cos \theta, r \sin \theta) e^{-i \ell \theta} d\theta.
\]
Then, if $N \leq r \leq N+N^{-1}$, it is easy to see
\[
|(\mathbf{1}_{C_N})_{\ell} (r)| \sim N^{-1}, \ (\mathrm{if} \ |\ell| \leq 2^{-3} N), \quad 
|(\mathbf{1}_{C_N})_{\ell} (r)| \lesssim |\ell|^{-1}, \ (\mathrm{if} \ |\ell| \gtrsim N).
\]
Since $0 \leq \sigma <1/2$, these inequalities imply
\begin{align}
\|f_{N,\delta} \|_{H^{s,\sigma}}& \sim N^{s} \| \langle \Omega \rangle^{\sigma} \F_{x} f_{N,\delta}\|_{L_{\xi}^2} \notag \\
& = \delta N^{s- s_{\sigma} +\frac12 - \sigma} \Bigl(\int_0^{\infty} \| \langle \Omega \rangle^{\sigma} \mathbf{1}_{C_N}(r \cos \theta, r \sin \theta) \|_{L_{\theta}^2}^2 r dr\Bigr)^{\frac12} \notag \\
& \sim \delta N^{s-s_{\sigma} + \frac12 - \sigma} \Bigl(\sum_{\ell \in \Z} \langle \ell \rangle^{2\sigma} | (\mathbf{1}_{C_N})_{\ell} (N)|^2\Bigr)^{\frac12} \sim \delta N^{s-s_{\sigma}}.
\label{est:thm2-00}
\end{align}
In particular, it holds that $\|f_{N,\delta}\|_{H^{s_{\sigma}, \sigma}} \sim \delta$. 

Next, we will see that, for $t \in[0,2^{-5}]$, it holds that
\begin{equation}\label{est:thm2-01}
\Bigl\| \int_0^t e^{-i(t-t') \Delta} |e^{-it' \Delta} f_{N,\delta}|^2 d t' \Bigr\|_{H^{s,\sigma}} 
\gtrsim t \delta^2.
\end{equation}
For $\xi \in \R^2$ and $\widetilde{\xi} \in \R^2$, we define 
\begin{equation*}
\Phi(\xi, \widetilde{\xi}) =  |\xi|^2 -  |\xi - \widetilde{\xi}|^2 +  |\widetilde{\xi}|^2.
\end{equation*}
If $|\xi| < 1$, $\widetilde{\xi} \in C_N$, and $\xi - \widetilde{\xi} \in C_N$, we compute that
\begin{equation}
\label{est:modulation-thm1.3}
|\Phi(\xi, \widetilde{\xi})| \leq |\xi|^2 + \bigl| \bigl( |\widetilde{\xi}| + |\xi - \widetilde{\xi}| \bigr) \bigl( |\widetilde{\xi}| - |\xi-\widetilde{\xi}| \bigr) \bigr| \leq 2^3.
\end{equation}
Hence, if $t \in [0,2^{-5}]$, letting $B_1 = \{\xi \in \R^2 \, | \,  |\xi| <1\}$, we have
\begin{align*}
& \Bigl\| \int_0^t e^{-i(t-t') \Delta} |e^{-it' \Delta} f_{N,\delta}|^2 d t' \Bigr\|_{H^{s,\sigma}} 
\gtrsim \Bigl\| \langle \Omega \rangle^{\sigma}\F_{x} \Bigl(\int_0^t e^{-i(t-t') \Delta} |e^{-it' \Delta} f_{N,\delta}|^2 d t' \Bigr) \Bigr\|_{L_{\xi}^2(B_1)} \\
& = \delta^2 N^{1 - 2s_{\sigma} -2 \sigma} \Bigl\| \langle \Omega \rangle^{\sigma} \Bigl( \int^t_0 \int_{\R^2}  
e^{-i t' \Phi(\xi,\widetilde{\xi})} {\bf 1}_{C_N}(\xi-\widetilde{\xi}) {\bf 1}_{C_N}(-\widetilde{\xi}) d \widetilde{\xi} d t' \Bigr)\Bigr\|_{L_{\xi}^2(B_1)}\\
& \sim t \delta^2 N^{1 - 2s_{\sigma} -2 \sigma}\Bigl\| \langle \Omega \rangle^{\sigma} \Bigl( \int_{\R^2}  {\bf 1}_{C_N}(\xi-\widetilde{\xi}) {\bf 1}_{C_N}(-\widetilde{\xi}) d \widetilde{\xi} \Bigr)\Bigr\|_{L_{\xi}^2(B_1)}.
\end{align*}
Thus, \eqref{est:thm2-01} follows from
\begin{equation}\label{est:thm2-02}
\Bigl\| \langle \Omega \rangle^{\sigma} \Bigl( \int_{\R^2}   {\bf 1}_{C_N}(\xi-\widetilde{\xi}) {\bf 1}_{C_N}(-\widetilde{\xi}) d \widetilde{\xi} \Bigr)\Bigr\|_{L_{\xi}^2(B_1)} \gtrsim N^{-\frac32 + \sigma}.
\end{equation}
Let us write $\xi = (\xi^{(1)}, \xi^{(2)}) \in \R^2$. We define
\[
T_N= \Bigl\{\xi=(\xi^{(1)},\xi^{(2)}) \in \R^2 \, \Bigl| \, |\xi^{(1)}| \leq \frac{1}{4N}, \quad |\xi^{(2)}| \leq \frac12\Bigr\}.
\]
Put $F(\xi) = \int_{\R^2}  {\bf 1}_{C_N}(\xi-\widetilde{\xi}) {\bf 1}_{C_N}(-\widetilde{\xi}) d \widetilde{\xi}$. Clearly, $F$ is a real and positive-valued function which satisfies $F(\xi^{(1)},\xi^{(2)}) = F(\xi^{(1)}, -\xi^{(2)})$. 
Let us first observe that
\begin{equation}\label{est:FbelowN}
F(\xi) \gtrsim N^{-1} \quad \mathrm{if} \ \ \xi \in T_N.
\end{equation}
For $\xi_0 \in \R^2$, we define the set $T_N + \xi_0$ by the collection of $\xi \in \R^2$ such that $\xi - \xi_0 \in T_N$. 
It is easy to see that the following inclusion holds:
\[
T_{N} + \Bigl(N+\frac{1}{4N},0\Bigr) \subset C_N,
\]
which implies that if $\xi \in T_N$, we have
\begin{align*}
F(\xi) & \geq \int_{\R^2} {\bf 1}_{T_{N} + (N+\frac{1}{4N},0)}(\xi-\widetilde{\xi}) {\bf 1}_{T_{N} + (N+\frac{1}{4N},0)}(-\widetilde{\xi}) d \widetilde{\xi}\\
& = \int_{\R^2} {\bf 1}_{T_{N}}(\xi-\widetilde{\xi}) {\bf 1}_{T_{N}}(-\widetilde{\xi}) d \widetilde{\xi} = |T_N \cap (T_N + \xi)| \sim N^{-1}.
\end{align*}
This completes the proof of \eqref{est:FbelowN}. 

It is straightforward to check that if $|\xi^{(1)}| \geq 3/N$ then $F(\xi)=F(\xi^{(1)},\xi^{(2)})=0$. 
This and \eqref{est:FbelowN} suggest that if $1/4 \leq r \leq 1/2$ then we have
\begin{align*}
& F(r \cos \theta, r \sin \theta) \gtrsim N^{-1} \quad \mathrm{if} \ \ \Bigl| \theta - \frac{\pi}{2}\Bigr| \leq N^{-1},\\
& F(r \cos \theta, r \sin \theta) =0 \qquad \ \, \mathrm{if} \ \ \Bigl| \theta - \frac{\pi}{2} \Bigr| \geq 2^3 N^{-1} \ \mathrm{and} \ \theta \in [0,\pi].
\end{align*}
Note that, if $|\theta- \pi /2| \leq 2^3 N^{-1}$ and $\ell \in 2 \Z = \{ 0, \pm 2, \pm 4, \ldots\}$ satisfies $|\ell| \leq 2^{-10}N$, then the inequality $\cos (\ell \theta) > 1/2$ holds. 
Therefore, if $1/4 \leq r \leq 1/2$, for $\ell \in 2 \Z $ such that $|\ell| \leq 2^{-10} N$, it holds that
\begin{align*}
|(F)_{\ell}(r)| & = \frac{1}{2 \pi}\Bigl| \int_{0}^{2\pi} F(r \cos \theta, r \sin \theta) e^{-i \ell \theta} d \theta \Bigr| \\
& \geq \frac{1}{2\pi} \Bigl| \int_{0}^{2\pi} F(r \cos \theta, r \sin \theta) \cos (\ell \theta) d \theta \Bigr|\\
& = \frac{1}{\pi} \Bigl| \int_{0}^{\pi} F(r \cos \theta, r \sin \theta) \cos (\ell \theta) d \theta \Bigr| \gtrsim N^{-2}.
\end{align*}
Here we used $F(\xi^{(1)},\xi^{(2)}) = F(\xi^{(1)}, -\xi^{(2)})$ and $\cos (\ell \theta) = \cos (- \ell \theta)$. 
Consequently, we conclude that
\begin{align*}
\| \langle \Omega \rangle^{\sigma} F \|_{L_{\xi}^2(B_1)} & \gtrsim 
\Bigl( \int_{\frac14}^{\frac12} \int_0^{2 \pi} |\langle \Omega \rangle^{\sigma}F(r \cos\theta, r\sin \theta)|^2 d \theta r dr \Bigr)^{\frac12}\\
& \gtrsim \inf_{r\in [\frac{1}{4},\frac{1}{2}]}  \bigl( \sum_{\ell \in 2 \Z, |\ell| \leq 2^{-10}N} \langle \ell \rangle^{2 \sigma}|(F)_{\ell}(r)|^2 \bigr)^{\frac12}\\
& \sim N^{-\frac32 + \sigma},
\end{align*}
which implies \eqref{est:thm2-02}.

Now we prove Theorem \ref{thm2}. For $f \in L^2(\R^2)$, let us use
\[
\mathcal{Q}(f)(t) = - i \eta \int_0^t e^{-i(t-t') \Delta} |e^{-it' \Delta} f|^2 d t'.
\]
Since $s_{\sigma} > -1/2$, $\sigma = -2 s_{\sigma} -1/2$, $\|f_{N,\delta}\|_{H^{s_{\sigma}, \sigma}} \sim \delta$, and $\delta$ is small, it follows from Theorem \ref{thm1} that we may find the solution of \eqref{NLS} with the initial datum $f_{N,\delta}$ which satisfies
\begin{equation}
\label{solution-thm1.3-01}
u_{N,\delta}(t)  = e^{-it \Delta} f_{N,\delta} - i \eta \int_0^t e^{-i(t-t') \Delta} |u_{N,\delta}(t')|^2 d t', \quad t \in [0,1].
\end{equation}
We define $v_{N,\delta}$ as 
\[
v_{N,\delta}(t)  = u_{N,\delta}(t) -  e^{-it \Delta} f_{N,\delta}  -\mathcal{Q}(f_{N,\delta})(t), \quad t \in [0,1].
\]
Then, it follows from \eqref{solution-thm1.3-01} that 
\begin{align*}
v_{N,\delta}(t) & = - i \eta \int_0^t e^{-i(t-t') \Delta} |u_{N,\delta}(t')|^2 d t' -\mathcal{Q}(f_{N,\delta})(t) \\
& = - i \eta \int_0^t e^{-i(t-t') \Delta} \Bigl( \bigl| e^{-it' \Delta} f_{N,\delta}  +\mathcal{Q}(f_{N,\delta})(t')+v_{N,\delta}(t')\bigr|^2 - |e^{-it' \Delta} f_{N,\delta}|^2\Bigr) d t'\\
& = - i \eta \int_0^t e^{-i(t-t') \Delta} \Bigl( |v_{N,\delta}(t')|^2 + 2 \Re \bigl( (e^{-it' \Delta} f_{N, \delta}) \overline{v_{N,\delta}}(t') \bigr)\\
& \qquad \qquad \qquad \qquad + 2 \Re \bigl(  \overline{\mathcal{Q}(f_{N,\delta})}(t')(e^{-it' \Delta} f_{N,\delta}  + v_{N,\delta}(t'))\bigr) \Bigr) d t'.
\end{align*}
Since $s_{\sigma} > -1/2$, $\sigma = -2 s_{\sigma} -1/2$, as in the proof of~\eqref{est:NonlinearEst}, we may choose $0<\kappa<1$ so that the estimate
\begin{equation}\label{est:NonlinearEst-thm2}
\|\mathcal{N}(u,v)\|_{Z^{s_{\sigma},\sigma}_{\kappa}([0,1])}  \lesssim \|u\|_{Z^{s_{\sigma},\sigma}_{\kappa}([0,1])} \|v\|_{Z^{s_{\sigma},\sigma}_{\kappa}([0,1])}
\end{equation}
holds. Here the bilinear operator $\mathcal{N}(u,v)$ is defined in \eqref{definition:N}. 
Let $V(\delta) = \|v_{N,\delta}\|_{Z^{s_{\sigma},\sigma}_{\kappa}([0,1])}$ where $Z^{s_{\sigma},\sigma}_{\kappa}([0,1])$ is defined in Subsection~\ref{subsection2.1}. 
Then, by employing Lemma~\ref{lemma2.4}, \eqref{est:NonlinearEst-thm2} and \eqref{est:thm2-00} to estimate the integral form of $v_{N,\delta}$, we get
\[
V(\delta) \lesssim V(\delta)^2 + (\delta + \delta^2) V(\delta) + (\delta^3 + \delta^4).
\]
We recall the embedding \eqref{est:embeddingZ}. Because $V(\delta)$ is continuous and $V(0)=0$, we have
\begin{equation}
\label{est-thm2-03}
\| v_{N,\delta}(t) \|_{H^{s_{\sigma}, \sigma}} \leq V(\delta) \lesssim \delta^3,
\end{equation}
for any $t \in [0,1]$ and sufficiently small $\delta>0$. Then, by utilizing the estimates \eqref{est:thm2-01} and \eqref{est-thm2-03}, for any $t \in [0, 2^{-5}]$ there exists sufficiently small $\delta>0$ such that
\begin{align*}
\|u_{N,\delta}(t)\|_{H^{s,\sigma}} & \geq \|\mathcal{Q}(f_{N,\delta}) (t)\|_{H^{s,\sigma}} - \| v_{N,\delta}(t) \|_{H^{s_{\sigma}, \sigma}} -  \| f_{N,\delta} \|_{H^{s,\sigma}} \\
& \gtrsim t \delta^2.
\end{align*}
Because $\|f_{N,\delta} \|_{H^{s,\sigma}} \to 0$ as $N \to \infty$ 
which is a consequence of \eqref{est:thm2-00} and $s- s_{\sigma}<0$, this inequality implies the discontinuity of the flow map.
\end{proof}
Next, we prove Theorem~\ref{thm3}.
\begin{proof}[Proof of Theorem~\ref{thm3}]
Let $B_1 = \{\xi \in \R^2 \, | \,  |\xi| <1\}$. 
We will prove that, for arbitrarily large $C \gg 1$, there exists a radial function $f \in H^{-\frac12}_{\mathrm{rad}}(\R^2)$ such that
\begin{equation}\label{est:goal-thm1.3}
\sup_{0<t \leq 2^{-5}} \Bigl\| {\bf 1}_{B_1} (\xi) \F_{x}\Bigl( \int^t_0 e^{-i (t- t') \Delta} \bigl(
( e^{-it' \Delta} f ) (\overline{e^{-it'\Delta} f} )\bigr) d t'\Bigr) \Bigr\|_{L_{\xi}^2} \geq C \| f\|_{H^{-\frac12}}^2.
\end{equation}
Clearly, for any $s \leq  -1/2$, this yields
\[
\sup_{0<t \leq 2^{-5}} \Bigl\| \int^t_0 e^{-i (t- t') \Delta} \bigl(
( e^{-it' \Delta} f ) (\overline{e^{-it'\Delta} f} )\bigr) d t' \Bigr\|_{H^{s}} \geq C  \| f\|_{H^{s}}^2,
\]
which implies Theorem \ref{thm3}. For the details of the lack of the twice differentiability of the flow map, we refer to Section 6 of \cite{Bo97}.

Let $N$, $M$ be dyadic, and $0 \leq j \leq N-1$. We define the set
\[
A_N  = \{ \xi \in \R^2 \, | \, N \leq |\xi| \leq N+ N^{-1} \},
\]
and the radial function $f_N$ as
\begin{equation*}
\F_x{f_N} \, = \, N^{\frac12} {\bf{1}}_{A_N}.
\end{equation*}
Clearly, $\|f_N\|_{H^{-\frac12}(\R^2)} \sim 1$. 
For $\xi \in \R^2$ and $\widetilde{\xi} \in \R^2$, we define 
\begin{equation*}
\Phi(\xi, \widetilde{\xi}) =  |\xi|^2 -  |\xi - \widetilde{\xi}|^2 +  |\widetilde{\xi}|^2.
\end{equation*}
As in the proof of~\eqref{est:modulation-thm1.3}, we can observe that $|\Phi(\xi, \widetilde{\xi})| \leq 2^3$ if $|\xi| \leq 1$, $\widetilde{\xi} \in A_N$, and $\widetilde{\xi} - \xi \in A_N$. 
We will show that if $0 < t \leq 2^{-5}$, it holds that
\begin{equation}
\Bigl\| {\bf 1}_{B_1}(\xi) \F_{x} \Bigl( \int^t_0 e^{-i (t- t') \Delta} |e^{-it' \Delta} f_N |^2 d t' \Bigr)\Bigr\|_{L_{\xi}^2} \gtrsim t (\log N)^{\frac12}.\label{desired-est}
\end{equation}
For arbitrarily large $C \gg 1$, \eqref{est:goal-thm1.3} follows from \eqref{desired-est} by taking $N$ large. 
Applying the Fourier transform, \eqref{desired-est} is given by
\begin{equation}
\label{est:desired-est-2-thm1.5}
\Bigl\| {\bf 1}_{B_1}(\xi) \int^t_0 \int_{\R^2}  e^{-i t' \Phi(\xi,\widetilde{\xi})} {\bf 1}_{A_N}(\xi-\widetilde{\xi}) {\bf 1}_{A_N}(-\widetilde{\xi}) d \widetilde{\xi} d t' \Bigr\|_{L_{\xi}^2} \gtrsim t (\log N)^{\frac12}N^{-1}.
\end{equation}
Since $|\Phi(\xi, \widetilde{\xi})| \leq 2^3$ and $0<t \leq 2^{-5}$, we see $1/2 \leq \Re{\bigl(e^{i t' \Phi(\xi, \widetilde{\xi})}\bigr)}$ for any $t' \in [0,t]$. Hence, \eqref{est:desired-est-2-thm1.5} is equivalent to
\begin{equation}
\label{est:desired-est-3-thm1.5}
\Bigl\| {\bf 1}_{B_1}(\xi) \int_{\R^2}   {\bf 1}_{A_N}(\xi-\widetilde{\xi}) {\bf 1}_{A_N}(-\widetilde{\xi}) d \widetilde{\xi}  \Bigr\|_{L_{\xi}^2} \gtrsim  (\log N)^{\frac12} N^{-1}.
\end{equation}
Let $j \in 2^{\N_0}$ satisfy $0 \leq j \leq N-1$. We define
\[
{\mathfrak{D}}_{N,j}  = \Bigl\{ \xi = (r \cos \theta, r \sin \theta) \in A_N \, \Bigl| \, 
\theta \in \Bigl[\frac{2\pi}{N} j, \, \frac{2\pi}{N} \, (j+1) \Bigr]  \Bigr\}.
\]
Then, it follows from $A_N = \bigcup_{0 \leq j \leq N-1}{\mathfrak{D}}_{N,j}$ that
\begin{equation}
\label{est:thm1.5-01}
\begin{split}
& \Bigl\| {\bf 1}_{B_1}(\xi) \int_{\R^2}   {\bf 1}_{A_N}(\xi-\widetilde{\xi}) {\bf 1}_{A_N}(-\widetilde{\xi}) d \widetilde{\xi}  \Bigr\|_{L_{\xi}^2} \\
\geq & \Bigl\|  \sum_{0 \leq j \leq N-1} {\bf 1}_{B_1}(\xi) \int_{\R^2}  {\bf 1}_{{\mathfrak{D}}_{N,j} }(\xi-\widetilde{\xi}) {\bf 1}_{{\mathfrak{D}}_{N,j} }(-\widetilde{\xi}) d \widetilde{\xi}  \Bigr\|_{L_{\xi}^2}.
\end{split}
\end{equation}
We define the line segments $\{L_j\}$ and the tube sets $\{T_{N,j}\}$ by
\begin{align*}
L_j & = \Bigl\{  \Bigl( - r\sin \frac{2 \pi}{N}j, \, r \cos \frac{2 \pi}{N}j \Bigr) \in \R^2 \, \Bigl| \, |r| \leq 2^{-5}\Bigr\},\\
T_{N,j} & = \{\xi \in \R^2 \, | \, \inf_{\eta \in L_{j}} |\xi-\eta| \leq 2^{-5}N^{-1}\}.
\end{align*}
By the same argument as in the proof of \eqref{est:FbelowN}, we may obtain
\[
{\bf 1}_{B_1}(\xi) \int_{\R^2}  {\bf 1}_{{\mathfrak{D}}_{N,j} }(\xi-\widetilde{\xi}) {\bf 1}_{{\mathfrak{D}}_{N,j} }(-\widetilde{\xi}) d \widetilde{\xi}  \gtrsim N^{-1}  {\bf 1}_{T_{N,j}}(\xi).
\]
Therefore, combined with \eqref{est:thm1.5-01}, the estimate \eqref{est:desired-est-3-thm1.5} is verified by
\begin{equation}
\label{est:thm1.5-Kakeya}
\Bigl\| \sum_{0 \leq j \leq N-1} {\bf 1}_{T_{N,j}} \Bigr\|_{L_{\xi}^2} \sim (\log N)^{\frac12}.
\end{equation}
For $0 \leq j,k \leq N-1$, it is easy to see
\[
|T_{N,j} \cap T_{N,k}| \sim \frac{1}{N(1+|j-k|)}.
\]
By utilizing this, we have
\begin{align*}
\Bigl\| \sum_{0 \leq j \leq N-1} {\bf 1}_{T_{N,j}} \Bigr\|_{L_{\xi}^2}^2 & = 
\int_{\R^2} \sum_{0 \leq j,k  \leq N-1} {\bf 1}_{T_{N,j}}(\xi) {\bf 1}_{T_{N,k}}(\xi) d \xi\\
& \sim \sum_{0 \leq j,k  \leq N-1}\frac{1}{N(1+|j-k|)}\\
& \sim \sum_{0 \leq \ell \leq N-1} \frac{1}{(1+\ell)} \sim \log N.
\end{align*}
This completes the proof of \eqref{est:thm1.5-Kakeya}.
\end{proof}
\begin{rem}
The estimate \eqref{est:thm1.5-Kakeya} is connected with the Kakeya maximal estimate in $\R^2$ obtained by C\'{o}rdoba \cite{Cor77}.
\end{rem}
Next, we establish Proposition~\ref{proposition2.2}. The proof is based on that by Nakanishi, Takaoka, and Tsutsumi. See the proof of Theorem~1 (iii) in~\cite{NTT01}.
\begin{proof}[Proof of Proposition~\ref{proposition2.2}]
Let $N \gg 1$. We define the set $D_N \subset \R^3$ as
\[
D_N = \{ (\tau,|\xi|\cos\theta, |\xi|\sin \theta)  \, | \, N \leq |\xi| \leq N+1, \ |\theta| \leq N^{-1}, \ |\tau-|\xi|^2| \leq 1 \}.
\]
For $(\tau,\xi) \in \R^3$, we define $u_N$ by
\[
\F_{t,x}{u}_N (\tau,\xi)= N^{s + \frac12} \mathbf{1}_{D_N}(\tau,\xi).
\]
As in the proof of \eqref{est:thm2-00}, we have
\[
\|u_N\|_{X^{s,-2s-\frac12;b}} \sim N^{2 s+ \frac12} \|\langle \Omega \rangle^{-2 s - \frac12} \mathbf{1}_{D_N}(\tau,|\xi|\cos\theta, |\xi| \sin \theta) \|_{L_{\tau}^2 L_{\xi}^2} \sim 1.
\]
Thus, it suffices to prove
\begin{equation}\label{est:lemma2.2-1}
\|u_N \overline{u_N}\|_{X^{s, - 2 s -\frac12;-\frac12}} \gtrsim (\log N)^{\frac12}.
\end{equation}
For fixed $\tau \in \R$, we define
\[
F_{\tau}(\xi) = \int \mathbf{1}_{D_N}(\tau-\widetilde{\tau},\xi-\widetilde{\xi}) \mathbf{1}_{D_N}(-\widetilde{\tau} , -\widetilde{\xi}) d \widetilde{\tau} d \widetilde{\xi}.
\]
To show \eqref{est:lemma2.2-1}, we prove that if $\tau$ satisfies $|\tau| \leq N/10$, it holds that
\begin{equation}\label{est:lemma2.2-2}
\|\langle \Omega \rangle^{-2s-\frac12} F_{\tau}\|_{L_{\xi}^2(B_1)} \gtrsim N^{-1-2 s},
\end{equation}
which readily yields \eqref{est:lemma2.2-1}. Let $\xi=(\xi^{(1)},\xi^{(2)})$. It is easily observed that if $\xi$ satisfies $|\tau- 2 N \xi^{(1)}| \leq 1/10$ and $|\xi| \leq 1/2$ then it holds that $F_{\tau}(\xi) \sim 1$. On the other hand, we have
\[
\supp F_{\tau} \subset \{ \xi=(\xi^{(1)},\xi^{(2)}) \in \R^2 \, | \, |\tau-2 N \xi^{(1)}| \leq 10, \ |\xi| \leq 10 \}.
\]
Thus, for $\tau$, $r$ satisfying $|\tau| \leq N/10$ and $1/4 \leq r \leq 1/2$, we may find $\theta_{\tau,r} \in [\pi/4,  3\pi/4]$ such that 
\begin{align*}
& F_{\tau}(r \cos \theta, r \sin \theta) \sim 1 \quad \mathrm{if} \ \ \theta \in [0 , \pi] \ \ \mathrm{and} \ \  |\theta - \theta_{\tau,r}| \leq 2^{-5} N^{-1},\\
& F_{\tau}(r \cos \theta, r \sin \theta) = 0 \quad \mathrm{if} \ \ \theta \in [0 , \pi] \ \ \mathrm{and} \ \  |\theta - \theta_{\tau,r}| \geq  2^{5} N^{-1}.
\end{align*}
Notice that $F_{\tau}(\xi^{(1)}, \xi^{(2)}) = F_{\tau}(\xi^{(1)}, - \xi^{(2)})$. 
We also have that if $\theta \in [\pi/4,  3\pi/4]$, $\ell \in \Z$, and $|\cos (\ell \theta)| \ll 1$, then $|\cos ((\ell+1) \theta)| \sim 1$. 
These imply that, for $\tau$, $r$ satisfying $|\tau| \leq N/10$ and $1/4 \leq r \leq 1/2$, if $|\ell|\leq 2^{-10}N$, we have
\begin{align*}
& |(F_{\tau})_{\ell}| + |(F_{\tau})_{\ell+1}| \\ 
\sim & \Bigl| \int_{0}^{2\pi} F(r \cos \theta, r \sin \theta) \cos (\ell \theta) d \theta \Bigr|
+\Bigl| \int_{0}^{2\pi} F(r \cos \theta, r \sin \theta) \cos ((\ell+1) \theta) d \theta \Bigr|\\
\sim & N^{-1}.
\end{align*}
Hence, if $|\tau| \leq N/10$, in the same way as in the proof of \eqref{est:thm2-02}, we can get \eqref{est:lemma2.2-2} and complete the proof.
\end{proof}
Lastly, we prove Proposition~\ref{prop:SharpnessFactor}. 
\begin{proof}[Proof of Proposition~\ref{prop:SharpnessFactor}]
Throughout the proof, we assume $0< \varepsilon \ll 1$. First, consider the case $2 \leq p < \infty$. 
Put
\begin{align*}
& T_1(\varepsilon) = \{(\tau,|\xi|\cos \theta, |\xi|\sin \theta) \in \R^3 \, | \, |\tau-|\xi|^2|\leq \varepsilon, \ 1 \leq |\xi| \leq 1+\varepsilon, \ |\theta| \leq \varepsilon^{\frac12}\},\\
& T_2 (\varepsilon)= -T_1(\varepsilon) = \{(\tau,\xi) \in \R^3 \, | \, (-\tau, -\xi) \in T_1(\varepsilon) \},\\
& T_3(\varepsilon)= \{(\tau,\xi^{(1)},\xi^{(2)}) \in \R^3 \, | \, |\tau-|\xi|^2|\leq 2^{-5}\varepsilon, \ |\xi^{(1)}| \leq 2^{-5} \varepsilon, \ 2^{-5} \varepsilon^{\frac12} \leq |\xi^{(2)}| \leq 2^{-4} \varepsilon^{\frac12} \}.
\end{align*}
If we put $S_1=S_3=\{(|\xi|^2,\xi) \, | \, \xi \in B_2\}$ and $S_2=\{(-|\xi|^2, \xi) \, | \, \xi \in B_2\}$ where $B_2=\{\xi \in \R^2 \, | \, |\xi| < 2\}$, the hypersurfaces $S_1$, $S_2$, $S_3$ satisfy \textit{Assumption}~\ref{assumption:hypersurface} and for $i=1,2,3$, 
$T_i(\varepsilon) \subset S_i(\varepsilon)$ hold. 
In addition, it is easy to see
\[
\|{\bf 1}_{T_1(\varepsilon)} \|_{L_{\tau,r}^2 L_{\theta}^{p}} \sim \varepsilon^{1+\frac{1}{2p}}, \quad \|{\bf 1}_{T_2(\varepsilon)}\|_{L_{\tau,\xi}^2} \sim \|{\bf 1}_{T_3(\varepsilon)}\|_{L_{\tau,\xi}^2} \sim \varepsilon^{\frac54}.
\]
Hence, it suffices to show
\begin{equation}\label{est:proofofrem3.3}
| {\bf 1}_{T_1(\varepsilon)} * {\bf 1}_{T_2(\varepsilon)} *{\bf 1}_{T_3(\varepsilon)}(0)| \sim \varepsilon^{5}.
\end{equation}
As in the proof of \eqref{est:FbelowN}, by the definitions of $T_1(\varepsilon)$, $T_2(\varepsilon)$, $T_3(\varepsilon)$, for any fixed $(\tau,\xi) \in T_3(\varepsilon)$, it follows that
\[
\int {\bf 1}_{T_1(\varepsilon)}(\widetilde{\tau},\widetilde{\xi}) {\bf 1}_{T_2(\varepsilon)}(-\tau - \widetilde{\tau}, -\xi - \widetilde{\xi}) d\widetilde{\tau} d \widetilde{\xi} \sim \varepsilon^{\frac52}.
\]
This implies \eqref{est:proofofrem3.3}.

Next, we consider the case $p= \infty$. Let $N$ be dyadic such that $1 \leq N \leq 2^{-5}\varepsilon^{-\frac12}$. Define
\begin{align*}
& A_1(\varepsilon) = \{(\tau,\xi) \in \R^3 \, | \, |\tau-|\xi|^2|\leq \varepsilon, \ 1 \leq |\xi| \leq 1+\varepsilon \},\\
& A_2 (\varepsilon)= \{(\tau,\xi) \in \R^3 \, | \, |\tau+|\xi|^2+2 \varepsilon^{\frac12}|\leq \varepsilon, \ 1 -\varepsilon^{\frac12} \leq |\xi| \leq 1-\varepsilon^{\frac12} +\varepsilon \},\\
& A_{3,N}(\varepsilon)= \{(\tau,\xi) \in \R^3 \, | \, |\tau-|\xi|^2|\leq 2^{-5} \varepsilon, \ \varepsilon^{\frac12} + N \varepsilon \leq |\xi| \leq \varepsilon^{\frac12} + 2 N \varepsilon\}.
\end{align*}
Let $S_1=S_3=\{(|\xi|^2,\xi) \, | \, \xi \in B_2\}$ and $\widetilde{S}_2 = \{(-|\xi|^2-2 \varepsilon^{\frac12}, \xi) \, | \, \xi \in B_2\}$. Clearly, the hypersurfaces $S_1$, $\widetilde{S}_2$, $S_3$ satisfy \textit{Assumption}~\ref{assumption:hypersurface} and $A_1(\varepsilon) \subset S_1(\varepsilon)$, $A_2(\varepsilon) \subset \widetilde{S}_2(\varepsilon)$, $\bigcup_{1 \leq N \leq 2^{-5}\varepsilon^{-\frac12}}A_{3,N}(\varepsilon) \subset S_3(\varepsilon)$. We define
\[
f_1 = {\bf 1}_{A_1(\varepsilon)}, \quad f_2={\bf 1}_{A_2(\varepsilon)}, \quad f_3= \sum_{1\leq N \leq 2^{-5} \varepsilon^{-\frac12}}N^{-\frac12}{\bf 1}_{A_{3,N}(\varepsilon)}.
\]
Since $|A_1(\varepsilon)| \sim |A_2(\varepsilon)| \sim \varepsilon^2$ and $|A_{3,N}(\varepsilon)| \sim N \varepsilon^{\frac52}$, we have
\[
\|f_1\|_{L_{\tau,r}^2 L_{\theta}^{\infty}} \sim  \|f_1\|_{L^2} \sim \|f_2\|_{L^2} \sim \varepsilon, \quad \|f_3\|_{L^2} \sim \langle \log \varepsilon \rangle^{\frac12} \varepsilon^{\frac54}.
\]
Therefore, it suffices to prove
\begin{equation}\label{est:proofofrem3.3-2}
| {\bf 1}_{A_1(\varepsilon)} * {\bf 1}_{A_2(\varepsilon)} *{\bf 1}_{A_{3,N}(\varepsilon)}(0)| \sim N^{\frac12} \varepsilon^{\frac{19}{4}}.
\end{equation}
For fixed $(\tau,\xi) \in A_{3,N}(\varepsilon)$, we will show
\begin{equation}\label{est:proofofrem3.3-3}
\int {\bf 1}_{A_1(\varepsilon)}(\widetilde{\tau},\widetilde{\xi}) {\bf 1}_{A_2(\varepsilon)}(-\tau - \widetilde{\tau}, -\xi - \widetilde{\xi}) d\widetilde{\tau} d \widetilde{\xi} \sim N^{-\frac12} \varepsilon^{\frac94}.
\end{equation}
This readily gives \eqref{est:proofofrem3.3-2}. We consider \eqref{est:proofofrem3.3-3}. 
The proof is analogous to the geometric observation in \cite[Section 5.2.1]{Bej06}. 
For $(\tau,\xi) \in A_{3,N}(\varepsilon)$, put
\[
E(\tau,\xi) = \{(\widetilde{\tau},\widetilde{\xi})\in A_1(\varepsilon) \, | \, (-\tau-\widetilde{\tau},-\xi-\widetilde{\xi}) \in A_2(\varepsilon) \}.
\]
We need to prove that the measure of $E(\tau,\xi)$ is comparable to $N^{-\frac12}\varepsilon^{\frac94}$. For simplicity, here we fix $\xi = (\varepsilon^{\frac12}+N\varepsilon,0)$ and ignore the time component. It suffices to show that the measure of the set
\[
\biggl\{ (r \cos \theta, r \sin \theta) \in \R^2  \, \biggl| \, 
\begin{aligned}
&1 \leq r \leq 1+\varepsilon, \\
& 1 - \varepsilon^{\frac12} \leq |(r \cos \theta-\varepsilon^{\frac12} - N \varepsilon, r \sin \theta)| \leq 1 -\varepsilon^{\frac12}+\varepsilon
 \end{aligned} \biggr\}
\]
is bounded from below by $N^{-\frac12} \varepsilon^{\frac54}$. 
By symmetry, we may assume $\theta \in [0,\pi]$. 

For fixed $r$ satisfies $1 \leq r \leq 1+ \varepsilon$, we define the function $F_{r} : [0,\pi] \to \R$ by
\[
F_{r}(\theta) = |(r \cos \theta-\varepsilon^{\frac12} - N \varepsilon, r \sin \theta)|.
\]
Clearly, the claim follows once we find $\theta_{N,r} \in (0,\pi)$ such that for any $\theta$ satisfies $|\theta - \theta_{N,r}| \leq 2^{-5} N^{-\frac12} \varepsilon^{\frac14}$, it holds that
\[
1 - \varepsilon^{\frac12} \leq F_r(\theta) \leq 1 -\varepsilon^{\frac12}+\varepsilon.
\]
$F_{r}$ is a smooth and monotonically increasing function and it holds that
\begin{align*}
& F_{r}(0) = r - \varepsilon^{\frac12} - N \varepsilon \leq 1 - \varepsilon^{\frac12},\\
& F_r\bigl( \frac{\pi}{4}\bigr) = \Bigl| \Bigl( r \cos \frac{\pi}{4} - \varepsilon^{\frac12} - N \varepsilon, r \sin \frac{\pi}{4} \Bigr) \Bigr| \geq 1 - \varepsilon^{\frac12} + \varepsilon.
\end{align*}
Hence, we may find $\theta_{N,r} \in(0,\pi/4)$ such that
\[
F_r(\theta_{N,r}) = 1- \varepsilon^{\frac12} + 2^{-1}\varepsilon.
\]

Next, let us observe $\theta_{N,r} \leq 4 N^{\frac12} \varepsilon^{\frac14}$. 
Since
\begin{equation}\label{est:proofofrem3.3-3.5}
(F_r(\theta))^2 = - 2 r (\varepsilon^{\frac12} + N \varepsilon) \cos \theta +r^2 + (\varepsilon^{\frac12} + N \varepsilon)^2,
\end{equation}
if we define $G_r(\theta) = (F_r(\theta))^2$, for $n \in \N$, we have
\[
G_r^{(2n-1)} (0) =0, \qquad  |G_r^{(2n)} (0)| \leq 3  \varepsilon^{\frac12}.
\]
Therefore, we get
\[
G_r(\theta_{N,r}) - G_r(0) = \sum_{n=1}^{\infty}\frac{G_r^{(2 n)}(0)}{(2n)!} \theta_{N,r}^{2n}.
\]
From $ |G_r^{(2n)} (0)| \leq 3  \varepsilon^{\frac12}$, this implies
\begin{equation}\label{est:proofofrem3.3-4}
\frac{\varepsilon^{\frac12}}{2} \theta_{N,r}^2 - \sum_{n=2}^{\infty}\frac{3 \varepsilon^{\frac12}}{(2n)!} \theta_{N,r}^{2n} 
\leq G_r(\theta_{N,r}) - G_r(0) \leq 2  (2^{-1}+N) \varepsilon.
\end{equation}
Here we used $F_r(\theta) \leq 1$ for any $\theta \in [0,\pi/4]$. 
Because $0 < \theta_{N,r} < \pi/4$, it follows that
\[
1- \sum_{n=2}^{\infty} \frac{3}{(2n)!} \theta_{N,r}^{2(n-1)} \geq \frac12.
\]
Therefore, it follows from \eqref{est:proofofrem3.3-4} that
\[
\frac{\varepsilon^{\frac12}}{4} \theta_{N,r}^2 \leq 2 (2^{-1}+N) \varepsilon,
\]
which implies $\theta_{N,r} \leq 4 N^{\frac12} \varepsilon^{\frac14}$.

The equality \eqref{est:proofofrem3.3-3.5} gives
\[
|G_{r}'(\theta_{N,r})| \leq 2^4 N^{\frac12} \varepsilon^{\frac34}, \qquad |G_r^{(n)}(\theta_{N,r})| \leq 3 \varepsilon^{\frac12} \quad \mathrm{for} \ n \in \N.
\]
As a result, if $|\theta -\theta_{N,r}| \leq 2^{-5} N^{-\frac12} \varepsilon^{\frac14}$, the following Taylor series expansion of $G_r$ at $\theta_{N,r}$
\[
G_r(\theta) = G_r(\theta_{N,r}) + \sum_{n=1}^{\infty}\frac{G_r^{(n)}(\theta_{N,r})}{n!} (\theta - \theta_{N,r})^{n}.
\]
yields
\[
1 - \varepsilon^{\frac12} \leq F_r(\theta) \leq 1- \varepsilon^{\frac12} + \varepsilon,
\]
as desired.
\end{proof}

\section*{acknowledgements}
The authors would like to thank Takamori Kato for helpful suggestions related to Theorem~\ref{thm2} and Nobu Kishimoto for letting us know about his master's thesis \cite{Kishi08}. 
This work was supported by 
JSPS KAKENHI Grant Numbers
JP17K14220,
JP20K14342,
and JP21J00514.

\end{document}